\RequirePackage{fix-cm}
\documentclass[smallextended]{svjour3}       % onecolumn (second format)
\smartqed  % flush right qed marks, e.g. at end of proof
\usepackage{graphicx}
\usepackage{amsmath,amssymb}
\usepackage{here}
\usepackage{booktabs}
\usepackage{arydshln}
\usepackage{mathtools}
\usepackage{color}
\usepackage{mathtools}
\usepackage{lineno,hyperref}
\usepackage{graphicx}
\usepackage{color}
\usepackage{here}
\usepackage{mathtools}
\usepackage{arydshln}
\makeatletter
  \def\@thmcountersep{.}
\makeatother
\newtheorem{theo}{Theorem}[section]

\newtheorem{lem}[theo]{Lemma}
\newtheorem{cor}[theo]{Corollary}
\newtheorem{rem}[theo]{Remark}

\newtheorem{fact}[theo]{Fact}
\newcommand{\ol}{\overline}
\newcommand{\mc}{\mathcal}
\newcommand{\mb}{\mathbb}

\newcommand{\Honezero}{H^1_0(\Omega)}

\newcommand{\Ltwo}{L^{2}(\Omega)}

\newcommand{\Lq}{L^{q}(\Omega)}

\newcommand{\ra}{\rightarrow}
\newcommand{\f}{\frac}

\newcommand{\supp}{\text{\rm supp\,}}

\newcommand{\Lpnorm}[2]{\left\|#1\right\|_{L^{p}(#2)}}
\newcommand{\Lqnorm}[2]{\left\|#1\right\|_{L^{q}(#2)}}
\newcommand{\usupp}{\supp u_{-} }
\newcommand{\vDomain}{D(v_c)}
\newcommand{\Omegam}{D(m)}
\newcommand{\lionsA}{A}
\newcommand{\lionsB}{B}
\newcommand{\ball}{\overline{B}(\hat{u},\rho)}
\newcommand{\oballrho}{B(\hat{u},\rho)}

\newcommand{\ten}[1]{\times 10^{#1}}

\newcommand{\argmax}{\mathop{\rm arg~max}\limits}
\newcommand{\maximize}{\mathop{\rm maximize}\limits}
\newcommand{\opD}[1]{\mathring{D}(#1)}
\newcommand{\lowerboundoflambda}{\underline{\lambda}}

\begin{document}
\title{A posteriori verification of the positivity of solutions to elliptic boundary value problems
\thanks{This work is supported by JSPS KAKENHI Grant Number 19K14601 and JST CREST Grant Number JPMJCR14D4.
All data generated or analyzed during this study are included in this published article.
}}
%\subtitle{Do you have a subtitle?\\ If so, write it here}
%\titlerunning{Short form of title}        % if too long for running head
\author{Kazuaki Tanaka \and Taisei Asai}
%\authorrunning{Short form of author list} % if too long for running head
\institute{K. Tanaka \at
              Institute for Mathematical Science, Waseda University, 3-4-1, Okubo, Shinjuku-ku, Tokyo 169-8555, Japan \\
              Tel.: +81-3-5286-2923
%              Fax: +123-45-678910\\
              \email{tanaka@ims.sci.waseda.ac.jp}           %  \\
%             \emph{Present address:} of F. Author  %  if needed
           \and
           T. Asai \at
              Graduate School of Fundamental Science and Engineering, Waseda University, 3-4-1 Okubo, Shinjuku-ku, Tokyo 169-8555, Japan
}
\date{Received: date / Accepted: date}
% The correct dates will be entered by the editor
\maketitle
\begin{abstract}
			The purpose of this paper is to develop a unified a posteriori method for verifying the positivity of solutions of elliptic boundary value problems by assuming neither $H^2$-regularity nor $ L^{\infty} $-error estimation, but only $ H^1_0 $-error estimation.
			In [J. Comput. Appl. Math, Vol. 370, (2020) 112647],
			we proposed two approaches to verify the positivity of solutions of several semilinear elliptic boundary value problems.
			However, some cases require $ L^{\infty} $-error estimation and, therefore, narrow applicability.
			In this paper, we extend one of the approaches and combine it with a priori error bounds for Laplacian eigenvalues to obtain a unified method that has wide application.
			We describe how to evaluate some constants required to verify the positivity of desired solutions.
			We apply our method to several problems, including those to which the previous method is not applicable.
			
%	These requirements may narrow the applicability of the methods because the existing approach provided in \cite[Theorem 1 and Corollary 1]{plum1992explicit} requires $ u,\hat{u} \in H^2(\Omega) $, evaluating the bound for the embedding $ H^2(\Omega) \hookrightarrow L^{\infty}(\Omega) $,
%to obtain $ \sigma $ from $ \rho $.
%The $H^2$-regularity of $u$ may fall when $\Omega$ is a nonconvex polygonal domain.
%
%	 The purpose of this paper is to develop a unified method for verifying the positivity of solutions $ u $ of \eqref{eq:mainpro} by assuming neither $H^2$-regularity nor \eqref{eq:linferror}, but only $ H^1_0 $-error estimation \eqref{eq:h10error},
%%while only assuming, 
%which can be applied to all the cases in Table \ref{table:applicable} for arbitrary bounded domains $\Omega$.

% \PACS{PACS code1 \and PACS code2 \and more}
\end{abstract}
	\section{Introduction}\label{sec:1}
	In recent decades, 
	numerical verification (also known as computer-assisted proof, validated numerics, or verified numerical computation) has been developed and applied to various partial differential equations, including those where purely analytical methods have failed (see, for example, \cite{plum1992explicit,day2007validated,plum2008,mckenna2009uniqueness,nakao2011numerical,mckenna2012computer,takayasu2013verified,nakaoplumwatanabe2019numerical,tanaka2020numerical} and the references therein).
	One such successful application is to the semilinear elliptic boundary value problem
	\begin{align}
		\label{eq:mainpro}
		\left\{
		\begin{array}{ll}
		-\Delta u=f(u)&\text{in}~\Omega,\\
		u=0&\text{on}~\partial\Omega,
		\end{array}
		\right.
	\end{align}
	where $\Omega \subset \mathbb{R}^{N}$~$(N=2,3,\cdots)$ is a given bounded domain,
	$\Delta $ is the Laplacian,
	and $ f: \mb{R} \ra \mb{R}$ is a given nonlinear map.
	Further regularity assumptions for $\Omega$ and $f$ will be shown later for our setting.
	
	Positive solutions of \eqref{eq:mainpro} have attracted significant attention \cite{lions1982existence,gidas1979symmetry,lin1994uniqueness,damascelli1999qualitative,gladiali2011bifurcation,de2019morse}.
	For example, positive solutions of problem \eqref{eq:mainpro} with $ f(t) = \lambda t + t|t|^{p-1} $, $ \lambda \in [0,\lambda_1(\Omega)) $, $p\in (1,p^*)$ have been investigated from various points of view such as uniqueness, multiplicity, nondegeneracy, and symmetry, \cite{gidas1979symmetry,lin1994uniqueness,damascelli1999qualitative,gladiali2011bifurcation,de2019morse},
	where $p^*=\infty$ when $N=2$ and $p^*=(N+2)/(N-2)$ when $ N\geq3 $;
	$ \lambda_1(\Omega) $ is the first eigenvalue of $ -\Delta $ with the homogeneous Dirichlet boundary value condition in the weak sense.
	Another important nonlinearity is $ f(t)=\lambda (t-t^3) $, $ \lambda>0 $.
	This corresponds to the stationary problem of the Allen–Cahn equation \cite{allen1979microscopic}.
	The problem \eqref{eq:mainpro} with this nonlinearity may have a positive solution when $ \lambda \geq \lambda_1(\Omega)$.
	However, when $ \lambda < \lambda_1(\Omega)$, no positive solution is admitted;
	this can be confirmed by multiplying $-\Delta u = \lambda(u-u^3)$ with the first eigenfunction of $-\Delta$ and integrating both sides.
%	this can be confirmed by multiplying $-\Delta u = \lambda(u-u^3)$ with the first eigenfunction of $-\Delta$ and integrating both sides.
	The Allen–Cahn equation is a special case of the Nagumo equation \cite{mckean1970nagumo} with the nonlinearity $ f(t)=\lambda t(1-t)(t-a) $, $ \lambda>0 $ and $ 0<a<1$, and both of these equations have been investigated by many researchers.
	%; see Subsection \ref{subsec:nagumo}.
	%These equations has been investigated by many researchers.
	We are moreover interested in the related case in which $ f(t)=\lambda (t+\lionsA t^2 - \lionsB t^3) $ with $ \lionsA,\lionsB>0 $.
	The bifurcation of problem \eqref{eq:mainpro} with this nonlinearity was analyzed in \cite{lions1982existence}.
	This problem has two positive solutions when $ \lambda^*< \lambda < \lambda_1(\Omega)$ for some $\lambda^*>0$.
	Despite these results, quantitative information about the positive solutions, such as their shape, has not been clarified analytically.
	Throughout this paper, $ H^k(\Omega) $ denotes the $k$th order $ L^2 $ Sobolev space.
	We define $ \Honezero := \{ u \in H^1(\Omega) : u = 0~\mbox{on} ~ \partial \Omega\} $, with the inner product $(u, v)_{H^1_0}:=(\nabla u, \nabla v)_{L^2}$ and norm $\| u \|_{H^1_0} := \sqrt{(u, u)_{\smash{H^1_0}}}$.
	
	Numerical verification methods enable us to obtain an explicit ball containing exact solutions of \eqref{eq:mainpro}.
	More precisely, for a numerical approximation $ \hat{u} \in \Honezero $ that satisfies the assumption required by such methods, they prove the existence of an exact solution $ u \in H^1_0(\Omega) $ of \eqref{eq:mainpro} that satisfies
	\begin{align}
		\label{eq:h10error}
		\left\|u-\hat{u}\right\|_{H_{0}^{1}} \leq \rho
	\end{align}
	with an explicit error bound $ \rho>0 $.
	Under an appropriate condition, we can obtain an $ L^{\infty} $-estimation
	\begin{align}
		\label{eq:linferror}
		\left\|u-\hat{u}\right\|_{L^{\infty}} \leq \sigma
	\end{align}
	with bound $ \sigma>0 $.
	For instance, when $ u,\hat{u} \in H^2(\Omega) $, we can evaluate the $ L^{\infty} $-bound $ \sigma>0 $ by considering the embedding $ H^2(\Omega) \hookrightarrow L^{\infty}(\Omega) $ \cite[Theorem 1 and Corollary 1]{plum1992explicit}.
	Thus, these approaches have the advantage that quantitative information about a target solution is provided accurately in a strict mathematical sense.
	We can identify the approximate shape of solutions from the error estimates.
	Despite these advantages, information about the positivity of solutions is not guaranteed without further considerations, irrespective of how small the error bound ($\rho$ or $\sigma$) is.
	In the homogeneous Dirichlet case \eqref{eq:mainpro}, it is possible for a solution that is verified by such methods to be negative near the boundary $\partial\Omega$.
	
	Therefore, we developed methods of verifying the positivity of solutions of \eqref{eq:mainpro} in previous studies \cite{tanaka2015numerical,tanaka2017numerical,tanaka2017sharp,tanaka2020numerical} and applied these result to sing-changing solutions \cite{tanaka2021posteriori}.
	These methods succeeded in verifying the existence of positive solutions by checking simple conditions.
	In \cite{tanaka2015numerical,tanaka2017numerical,tanaka2017sharp}, we proposed methods for verifying the positivity of solutions of \eqref{eq:mainpro} by assuming both error estimates \eqref{eq:h10error} and \eqref{eq:linferror}.
	Subsequently, in \cite{tanaka2020numerical}, we extended our method to a union of two different approaches \cite[Theorems 2.1 and 3.2]{tanaka2020numerical} under certain conditions for nonlinearity $ f $.
	Table \ref{table:applicable} summarizes the error-estimate types that are required by these theorems when $ f $ is a subcritical polynomial
	\begin{align}
	f(t) = \lambda t + \sum_{i=2}^{n(<p^*)}a_{i}t|t|^{i-1},~~\lambda,~a_i \in \mb{R},~a_i \neq 0 \text{~~for~some~~} i. \label{eq:polyf}
	\end{align}
	Theorem 2.1 in \cite{tanaka2020numerical} can be applied to cases in which $\lambda <\lambda_1(\Omega)$.
	This theorem is based on the constructive norm estimation for the minus part $ u_{-}:=\max\{-u,0\} $ of a solution $ u $, and does not assume an $ L^{\infty} $-estimation \eqref{eq:linferror} but only requires an $ H^1_0 $-error estimation \eqref{eq:h10error}. 
	When $a_i \leq 0$ for all $ i $ and $\lambda \geq \lambda_1(\Omega)$, we used a completely different approach \cite[Theorem 3.2]{tanaka2020numerical} that was based on the Newton iteration that retains nonnegativity.
	This theorem needs no $ L^{\infty} $-estimation but requires an explicit evaluation of the minimal eigenvalue of a certain linearized operator around approximation $ \hat{u} $.
	Actually, this eigenvalue evaluation itself is not trivial to obtain.
	In \cite{tanaka2020numerical}, the eigenvalue was estimated using the existing method \cite{tanaka2020numerical,liu2015framework} based on Galerkin approximations.
	Theorems 2.1 and 3.2 in \cite{tanaka2020numerical} were applied numerically to problem \eqref{eq:mainpro} with the above-mentioned nonlinearities: $ f(t) = \lambda t + t|t|^{p-1} $ and $ f(t) = \lambda (t-t^3) $.	
	However, a problem still remains in the sense that we need the $ L^{\infty} $-error estimation \eqref{eq:linferror} to prove positivity when $a_i a_j < 0$ for some $ i,j $ and $\lambda \geq \lambda_1(\Omega)$.
	For example, the nonlinearity $ f(t)=\lambda (t+\lionsA t^2 - \lionsB t^3) $ in which we are interested requires such an estimation.
	These requirements may narrow the applicability of the methods because the existing approach provided in \cite[Theorem 1 and Corollary 1]{plum1992explicit} requires $ u,\hat{u} \in H^2(\Omega) $, evaluating the bound for the embedding $ H^2(\Omega) \hookrightarrow L^{\infty}(\Omega) $,
	to obtain $ \sigma $ from $ \rho $.
	The $H^2$-regularity of $u$ may fall when $\Omega$ is a nonconvex polygonal domain.
	
	\renewcommand{\arraystretch}{1.3}
	\begin{table}[h]
		\label{table:applicable}
		\caption{Error estimates required in \cite{tanaka2020numerical} when $ f $ is the subcritical polynomial \eqref{eq:polyf}.}
		\begin{center}
			\begin{tabular}{lll}%{p{40mm}p{45mm}p{40mm}}
			\toprule
				$a_i$&$\lambda$&\\
				\cmidrule(r){2-3}
				&$\geq \lambda_1(\Omega)$&$<\lambda_1(\Omega)$\\
				\midrule
				$a_i \geq 0$ for all $ i $ &No positive solution& \eqref{eq:h10error}, \cite[Theo. 2.1]{tanaka2020numerical}\\
				\midrule
				$a_i \leq 0$ for all $ i $ &\eqref{eq:h10error}, \cite[Theo. 3.2]{tanaka2020numerical}& No positive solution\\
				\midrule
				$a_i a_j < 0$ for some $ i,j $ &\eqref{eq:h10error} and \eqref{eq:linferror}, \cite[Cor. A.1]{tanaka2020numerical}& \eqref{eq:h10error}, \cite[Theo. 2.1]{tanaka2020numerical}\\
				\bottomrule
			\end{tabular}
		\end{center}
	\end{table}
	\renewcommand{\arraystretch}{1}
		 
	 The purpose of this paper is to develop a unified method for verifying the positivity of solutions $ u $ of \eqref{eq:mainpro} by assuming neither $H^2$-regularity nor \eqref{eq:linferror}, but only $ H^1_0 $-error estimation \eqref{eq:h10error},
	 %while only assuming, 
	 which can be applied to all the cases in Table \ref{table:applicable} for arbitrary bounded domains $\Omega$.
	 % without assuming $H^2$-regularity.
	 Our method is based on a posteriori constructive norm estimation for the minus part $ u_{-} $ and can be regarded as an extension of \cite[Theorem 2.1]{tanaka2020numerical}.
	 In short, we confirm that the norm of $ u_{-} $ vanishes by checking certain inequalities while assuming \eqref{eq:h10error} (see Lemma \ref{theo1}).
	 One of the key points is to estimate lower bounds of eigenvalue $\lambda_{1}(\usupp)$ explicitly because the inequality $\lambda_1(\usupp) \geq \lambda$ has to be confirmed for the success of our positivity proof (again, see Lemma \ref{theo1}).
	 Here, $\usupp := \overline{\{x \in \Omega : u_{-}(x) \neq 0\}}$ is the support of $u_{-}$, and $\lambda_1(\usupp)$ is understood as the first eigenvalue on the interior of $\usupp$.
	 When the interior of $\usupp$ is empty, we interpret $\lambda_{1}(\usupp)=\infty$, and all real numbers are lower bounds of $\lambda_{1}(\usupp)$.
	 The difficulty is that we cannot identify the location and shape of $ \usupp $ from \eqref{eq:h10error} even when $\hat{u}$ is nonnegative.
	 
	If a polygon or polyhedron $S$ enclosing $ \usupp $ is obtained concretely,
	%which is realized when we assume $ L^{\infty} $-error \eqref{eq:linferror},
	we can apply the Liu–Oishi method \cite{liu2013verified,liu2015framework} based on finite element methods to obtain a lower bound for $\lambda_{1}(S)$.
	Then, the inequality $\lambda_{1}(\usupp) \geq \lambda_{1}(S)$ gives the desired lower bound of $\lambda_{1}(\usupp)$.
	Such a supremum set $S$ over $ \usupp $ can be obtained when we have an $ L^{\infty} $-estimation \eqref{eq:linferror}.
	By setting $\Omega_{+}:=\{x\in\Omega : \hat{u}-\sigma\geq 0\}$ where $ u \geq 0 $ therein,
	we can construct such a domain $S$ as a supremum set of $\Omega\backslash\Omega_{+}$.
	Again, we cannot determine such a supremum set $S$ only from $ H^1_0 $-error estimation \eqref{eq:h10error}.
	The Temple–Lehmann–Goerisch method can help us to evaluate $\lambda_1(\usupp)$ more accurately (see, for example, \cite[Theorem 10.31]{nakaoplumwatanabe2019numerical}).
	
	To estimate the lower bound of $\lambda_{1}(\usupp)$, we rely on the following argument:
	\begin{fact}
	\label{fact:lowerbound}
		For a bounded domain $\Omega \subset \mathbb{R}^{N}$ $(N=2,3,\cdots)$,
		there exists a constant $A_{k,N}$ independent of $\Omega$ such that
		\begin{align}
		\lambda_{k}(\Omega) \geq A_{k,N} \left(\frac{1}{|\Omega|}\right)^{\frac{2}{N}},
		\label{ineq:Akn}
		\end{align}
		where $ \lambda_{k}(\Omega) $ denotes the $ k $-th eigenvalue of $-\Delta$ with the homogeneous Dirichlet boundary condition.
	\end{fact}
	Many articles have investigated this type of inequality in several forms.
%	This type of inequality has been investigated in many forms.
	Among them, we mainly use the Rayleigh–Faber–Krahn inequality, which ensures Fact \ref{fact:lowerbound} for $k=1$ \cite{faber1923bweis,krahn1925uber,krahn1926uber}. This inequality states that if $|\Omega| = |\Omega^*|$ 
	($\Omega^*$ is a ball in $\mathbb{R}^{N}$), then $\lambda_{1}(\Omega) \geq \lambda_{1}(\Omega^*)$, where the equality holds if and only if $\Omega=\Omega^*$.
	We also refer to \cite{li1983schrodinger} for an easy-to-estimate formula for $A_{k,N}$ for all $ k \geq 1 $ (see Remark \ref{rem:li}).
	Section \ref{sec:lowerbound} provides explicit lower bounds of $A_{k,N}$ based on these results.
	%$v \in \{w \in \Honezero : \left\|w-\hat{u}\right\|_{H^1_0}\leq \rho\}$
	%Fact \ref{fact:lowerbound} ensures that $\lambda_{1}(\usupp)$ goes to infinity as $ \rho \downarrow 0 $ when $\hat{u}$ is continuous and positive.
%	Fact \ref{fact:lowerbound} ensures that $\lambda_{1}(\usupp)$ goes to infinity as a continuous approximation $\hat{u}$ approaches to a positive function $u$ and $ \rho \downarrow 0 $.
	To estimate a lower bound of $\lambda_1(\usupp)$ using the inequality \eqref{ineq:Akn}, we focus on estimating upper bounds of $|\usupp|$ while assuming only the $ H^1_0 $-error estimation \eqref{eq:h10error} without knowing the specific shape and location of $\usupp$.
	Suppose that \eqref{eq:h10error} is proved for a positive approximation $\hat{u}$ with sufficient accuracy.
	Then, the upper bound for $|\usupp|$ can be estimated very small using Lemma \ref{lem:dm} provided later,
	and therefore, $\lambda_1(\usupp) \geq \lambda$ can be confirmed using \eqref{ineq:Akn} for a moderately large $\lambda$.
%	When an approximation $\hat{u}$ of a positive solution $u$ is computed with sufficient accuracy,
%	Therefore, if $ H^1_0 $-error estimation \eqref{eq:h10error} succeeds for a positive approximation $\hat{u}$ with sufficient accuracy,
%	Fact \ref{fact:lowerbound} can confirm $\lambda_1(\usupp) \geq \lambda$ for a moderately large $\lambda$.
	The established estimation for $|\usupp|$ is used to evaluate not only $\lambda_{1}(\usupp)$ but also some Sobolev embedding constants on $|\usupp|$,
	which play an essential role for our positivity proof (again, see Lemma \ref{theo1})
	 
	 	 The remainder of this paper is organized as follows.
	 Section \ref{sec:pre} introduces required notation and definitions.
	 In Section \ref{sec:volminuspart}, we evaluate an upper bound for the volume $ |\usupp| $ assuming the $H^1_0$-estimation \eqref{eq:h10error} for a continuous or piecewise continuous approximation $\hat{u} \in H^1_0(\Omega)$.
	 Subsequently, we use the bound for $ |\usupp| $ to evaluate lower bounds for $ \lambda_1(\usupp) $ in Section \ref{sec:lowerbound}.
	 In Section \ref{sec:embedding}, required Sobolev embedding constants on bounded domains are evaluated.
	 In Section \ref{sec:positive}, we extend the previous formula \cite[Theorem 2.1]{tanaka2020numerical} and combine it with the estimates derived from Sections \ref{sec:lowerbound} and \ref{sec:embedding}, thereby designing a unified method for proving positivity.
	 Finally, Section \ref{sec:ex} presents numerical examples where the proposed method is applied to problem \eqref{eq:mainpro} with several nonlinearities, including those to which the previous method is not applicable without an $ L^{\infty} $-error estimation.
	 \section{Preliminaries}\label{sec:pre}
	We begin by introducing required notation.
	We denote by $ H^{-1}$ the topological dual of $ \Honezero $.
	When describing norms and inner products, we may omit the domain of a function space unless there is a risk of misunderstanding.
	For example, we simply write $\|\cdot\|_{L^p} = \|\cdot\|_{L^p(\Omega)}$ if no confusion arises.
	For two Banach spaces $X$ and $Y$, the set of bounded linear operators from $X$ to $Y$ is denoted by ${\mc L}(X, Y)$ with the usual supremum norm $\| T \|_{{\mc L}(X, Y)} := \sup\{ \| T u \|_{Y} / \| u \|_{X} : {u \in X \setminus \{0\}}\}  $ for $T \in {\mc L}(X, Y)$.
	The norm bound for the embedding $\Honezero \hookrightarrow L^{p+1}\left(\Omega\right)$ is denoted by $C_{p+1}(=C_{p+1}(\Omega))$; that is, $C_{p+1}$ is a positive number that satisfies
	\begin{align}
	\label{embedding}
	\left\|u\right\|_{L^{p+1}(\Omega)}\leq C_{p+1}\left\|u\right\|_{\Honezero}~~~{\rm for~all}~u\in \Honezero,
	\end{align}
	where $p\in [1,\infty)$ when $N=2$ and $p\in [1,p^*]$ when $ N\geq3 $.
	If no confusion arises, we use the notation $ C_{p+1} $ to represent the embedding constant on the entire domain $ \Omega $,
	whereas, in some parts of this paper, we need to consider an embedding constant on some subdomain $ \Omega'\subset\Omega $.
	This is denoted by $ C_{p+1}(\Omega') $ to avoid confusion.
	Moreover, $ \lambda_1(\Omega) $ denotes the first eigenvalue of $ -\Delta $ imposed on the homogeneous Dirichlet boundary condition.
	This is characterized by
	\begin{align}
	\label{eq:defi-diri-eigenvalue}
	\lambda_{1}(\Omega) = \inf_{v\in \Honezero^\backslash{\{0\}}} \frac{\|v\|_{\Honezero}^2}{\|v\|_{L^2(\Omega)}^2}.
	\end{align}		
	Throughout this paper, we assume that $ f $ is a $ C^1 $ function that satisfies
	\begin{align}
	&|f(t)| \leq a_0 |t|^p + b_0 \text{~~~for~all~~} t \in \mb{R},\label{asuume:f}\\
	&|f'(t)| \leq a_1 |t|^{p-1} + b_1 \text{~~~for~all~~} t \in \mb{R}\label{asuume:df}
	\end{align}
	for some $ a_0,a_1,b_0,b_1\geq 0 $ and $p\in [1,p^*)$.
	We define the operator $ F $ as
	\begin{align*}
	F : \left\{\begin{array}{ccc}{u(\cdot)} & {\mapsto} & {f(u(\cdot))}, \\
	{\Honezero} & {\rightarrow} & {H^{-1}}.\end{array}\right.
	\end{align*}
	Moreover, we define another operator $ \mathcal{F} : \Honezero \rightarrow H^{-1}$ as $\mathcal{F}(u) :=-\Delta u-F(u) $, which is characterized by
	\begin{align}
	\left<\mathcal F(u),v\right> = \left(\nabla u,\nabla v\right)_{L^2} - \left<F(u),v\right>  \text{~~for~all~~} u,v \in \Honezero,
	\end{align}
	where $\left<F(u),v\right> = \int_{\Omega} f(u(x)) v(x) dx$.
	The Fr\'echet derivatives of $ F $ and $ \mc{F} $ at $ \varphi \in \Honezero $, denoted by $ {F'_{\varphi}} $ and $ {\mc F'_{\varphi}} $, respectively, are given by
	\begin{align}
	&\langle F'_{\varphi}u,v\rangle = \int_{\Omega} f'(\varphi(x))u(x) v(x) dx \text{~~for~all~~} u,v \in \Honezero,\\
	&\langle \mathcal F'_{\varphi}u,v \rangle = \left(\nabla u,\nabla v\right)_{L^2} - \langle F'_{\varphi}u,v \rangle  \text{~~for~all~~} u,v \in \Honezero. \label{def:derivativecalf}
	\end{align}
	Under the notation and assumptions, we look for solutions $ u \in \Honezero $ of 
	\begin{align}
	\label{main:fpro}
	\mathcal{F}(u)=0,
	\end{align}
	which corresponds to the weak form of \eqref{eq:mainpro}.
	We assume that some verification method succeeds in proving the existence of a solution $u \in \Honezero $ of \eqref{main:fpro} satisfying inequality \eqref{eq:h10error}
	given $ \hat{u} \in \Honezero $ and $ \rho>0 $.
	Although the regularity assumption for $ \hat{u} $ (to be in $ \Honezero $) is sufficient to obtain the error bound \eqref{eq:h10error} in theory,
	we further assume that $ \hat{u} $ is continuous or piecewise continuous throughout this paper.
	This assumption impairs little of the flexibility of actual numerical computation methods.
	We recall $ u_{-} = \max\{-u,0\} $, and define $ u_{+} := \max\{u,0\} $.
	\section{Evaluation of the volume of $ \usupp $}
	\label{sec:volminuspart}
	To estimate an upper bound for $ |\usupp| $ from the information of the inclusion \eqref{eq:h10error}, we define $ D(v):=\{ x \in \Omega : \hat{u}(x)\leq v(x)\}$ and consider the maximization problem
	\begin{align}
		\label{problem:Dopt}
		\maximize_{\Lqnorm{v}{\Omega}= c} |D(v)|
	\end{align}
	for fixed $ q \in (1,\infty) $ and $ c>0 $.
%	\begin{align*}
%%		\label{problem:Dopt}
%		\argmax_{\Lqnorm{v}{\Omega}= c} |D(v)|,
%	\end{align*}
%	where $ D(v):=\{ x \in \Omega : \hat{u}(x)\leq v(x)\}$ and $|D(v)|$ denotes its volume.
	This maximization takes place over the set of all functions $ v \in \Lq $ satisfying $ \Lqnorm{v}{\Omega}= c $.
	When $ \Lqnorm{\hat{u}_{+}}{\Omega} \leq c $, the maximal value of the problem \eqref{problem:Dopt} is $ |\Omega| $.
	Therefore, we consider the case where $ \Lqnorm{\hat{u}_{+}}{\Omega} > c $.
	In the following, we denote $ D(l) := \{ x \in \Omega : \hat{u}(x)\leq l\} $ and $ \mathring{D}(l) := \{ x \in \Omega : \hat{u}(x)< l\} $ for $l \in \mathbb{R}$.
	
	%in the following lemma.
	%prepare the following lemma.
	\begin{lem}
		\label{lem:volumeofv}
		%Given $ q \in (1,\infty) $ and $ c>0 $,
%		For fixed $ q \in (1,\infty) $ and $ c>0 $, we have
		Let $ q \in (1,\infty) $ and $ c>0 $ be fixed.
		Suppose that $ \Lqnorm{\hat{u}_{+}}{\Omega} > c $.
		Then, we have
		%Let $ v_c \in \Lq $ be some function satisfying $ \Lqnorm{v_c}{\Omega}= c $.
		%The volume of
		%$D(v)$
		%$ \vDomain:=\{ x \in \Omega : \hat{u}(x)\leq v_c(x)\}$
		%is maximized when
		\begin{align}
		%v(x)
		\argmax_{\Lqnorm{v}{\Omega}= c} |D(v)|
		=
		\left\{
		\begin{array}{ll}
		\hat{u}_+(x),&x\in D,\\
		0,&\text{\rm otherwise},
		\end{array}
		\right.		
		\end{align}
		where $ D $ is a set that satisfies
		$\Lqnorm{\hat{u}_{+}}{D}= c$ and
		\begin{align}
			\label{eq:Dandm}
%			\subpart{\hat{u}}{l} \subseteq D \subseteq \subpartq{\hat{u}}{l}
			\opD{l} \subseteq D \subseteq D(l)
		\end{align}
		%これを満たすlは必ずあるのか？
		for some $ l \in \mb{R} $.
		The maximal value of the problem \eqref{problem:Dopt} is $|D|$.
		%Moreover $\maximize_{\Lqnorm{v}{\Omega}= c} |D(v)| = |D|$.
		%The maximal $ |D(v)| $ is $ |D| $.
	\end{lem}
	
	\begin{proof}
			%When $ \Lqnorm{u_{+}}{\Omega} \leq c $, the maximal $ \vDomain $ is $ \Omega $.
			%Therefore, we consider the case where $ \Lqnorm{u_{+}}{\Omega} > c $ in the following proof.
			Let us denote by $v_c$ an arbitrary function in $\{v \in \Lq : \Lqnorm{v}{\Omega}= c\}$.			
			Because $ \vDomain \subseteq D(|v_c|)$ and the equality holds when $ v_c \geq 0$ in $ \Omega $,
			the volume of $ \vDomain $ is maximized for a nonnegative $ v_c $.
			
			Let $v_c$ be nonnegative.
			If $ \hat{u}-v_c $ is strictly negative in some part $\Omega' \subset \supp v_c$ satisfying $|\Omega'| \neq 0$ and vanishes in $  (\supp v_c) \backslash \Omega'$,
			another $ v'_c \in \{v \in \Lq : \Lqnorm{v}{\Omega}= c\}$ with the same $ c>0 $ can be constructed to obtain larger $ D(v'_c) $ as follows.
			Since $ \Lqnorm{\hat{u}_{+}}{\Omega} > c = \left\|v_{c}\right\|_{L^{q}(\supp v_c)}$, we have
			\begin{align*}
				\left\|\hat{u}_{+}\right\|_{L^{q}(\Omega\backslash \supp v_c)}^{q}
				&> \left\|v_{c}\right\|_{L^{q}(\supp v_c)}^{q}-\left\|\hat{u}_{+}\right\|_{L^{q}(\supp v_c)}^{q}\\
				&= \displaystyle \int_{\supp v_c}\left(v_{c}(x)^{q}-\hat{u}_{+}(x)^{q}\right)dx\\
				%&= \displaystyle \int_{\Omega'}\left(v_{c}(x)^{q}-\hat{u}_{+}(x)^{q}\right)dx\\
				&=\left\|v_c\right\|^q_{L^{q}(\Omega')} - \left\|\hat{u}_{+}\right\|^q_{L^{q}(\Omega')}.
			\end{align*}
			Therefore, there exists $ \Omega'' \subset \Omega\backslash (\supp v_c) $ that satisfies 
			$\left\|\hat{u}_{+}\right\|^q_{L^{q}(\Omega'')}=\left\|v_c\right\|^q_{L^{q}(\Omega')} - \left\|\hat{u}_{+}\right\|^q_{L^{q}(\Omega')}$.
			Defining $ v'_c $ as
			\begin{align*}
			v'_c(x)=
			\left\{
			\begin{array}{ll}
			\hat{u}_+(x),&x\in \Omega' \cup \Omega'',\\
			v_c(x),&\text{otherwise},
			\end{array}
			\right.		
			\end{align*}
			we have 
			\begin{align*}
				\Lqnorm{v'_c}{\Omega}^q
				&= \Lqnorm{\hat{u}_{+}}{\Omega'}^q + \Lqnorm{\hat{u}_{+}}{\Omega''}^q + \Lqnorm{v_c}{\Omega \backslash (\Omega' \cup \Omega'')}^q\\
				&= \Lqnorm{\hat{u}_{+}}{\Omega'}^q -\left\|\hat{u}_{+}\right\|^q_{L^{q}(\Omega')} + \left\|v_c\right\|^q_{L^{q}(\Omega')} + \Lqnorm{v_c}{\Omega \backslash (\Omega' \cup \Omega'')}^q\\
				&= \Lqnorm{v_c}{\Omega \backslash \Omega''}^q = \Lqnorm{v_c}{\Omega}^q
			\end{align*}
			and $ \vDomain \subset D(v'_c)$ in the strict sense because $|\Omega''| \neq 0$.
			%$\Lqnorm{v'_c}{\Omega}= c$ and 
			Therefore, when $ |\vDomain| $ is maximized, $ \hat{u}-v_c $ vanishes in $ \supp v_c $; that is, 
			\begin{align}
			\label{cond:v}
			v_c(x)=\hat{u}_{+}(x),~~x \in \supp v_c.
			\end{align}
			
			Finally, we consider a subset $D\,(=\supp v_c) \subset \Omega$ with the largest volume satisfying $\Lqnorm{\hat{u}_{+}}{D}= c$.
			In the following, we prove that the volume of such $ D $ is maximized when \eqref{eq:Dandm} holds for some $ l \in \mb{R}$.
			%; this follows from the following fact.
			%The maximal volume of $ D $ satisfying $\Lqnorm{\hat{u}_{+}}{D}= c$ and \eqref{cond:v} is realized when \eqref{eq:Dandm} for some $ m \in \mb{R}$; this follows from the following fact.
			Since $\hat{u}$ is continuous or piecewise continuous on $\Omega$, there exist $D$ and $ l \in \mb{R}$ that satisfy $\Lqnorm{\hat{u}_{+}}{D}= c$ and \eqref{eq:Dandm}.
			Suppose that there exists a different set $D' \subset \Omega$ that satisfies $\Lqnorm{\hat{u}_{+}}{D'}= c$ 
			so that
			\begin{align}
				\label{eq:DDD}
				\Lqnorm{\hat{u}_{+}}{D'\backslash(D\cap D')} = \Lqnorm{\hat{u}_{+}}{D\backslash(D\cap D')}.
			\end{align}
			Then, we have that $\hat{u}_{+}(x)\leq l$ for all $x \in D\backslash(D\cap D') \subset D$ and $\hat{u}_{+}(x)\geq l$ for all $x \in D'\backslash(D\cap D') \subset \Omega \backslash D$ because \eqref{eq:Dandm}.
			It follows from \eqref{eq:DDD} that $|D\backslash(D\cap D')|\geq |D'\backslash(D\cap D')|$, and therefore, $|D|\geq |D'|$.
			Thus, the assertion of this lemma is proved.

%			$D=\opD{l} \cup (\text{a~subset~of~} \{x \in \Omega : \hat{u}_{+}(x)=l\})$

%		 ensures that $D=\opD{l} \cup (\text{a~subset~of~} \{x \in \Omega : \hat{u}_{+}(x)=l\})$.
			%If $D=D(l)$, then $\hat{u}_{+}(x)> l$ for all $x \in D'\backslash(D\cap D')$.
			%If $D \subset D(l)$ in the strict sense, $\hat{u}_{+}(x) \geq l$ for all $x \in D'\backslash(D\cap D')$.
%			Therefore, .

%			For two bounded domains $ \Omega_1 $ and $ \Omega_2 $, if the continuous or piecewise continuous functions $ u_1 $ and $ u_2 $ respectively over $ \Omega_1 $ and $ \Omega_2 $ satisfy $ \Lqnorm{u_1}{\Omega_1} = \Lqnorm{u_2}{\Omega_2} $ and $ u_1 \leq u_2 $ (i.e., $\max\{u_1(x) : x \in \Omega_1\} \leq \min\{u_2(x) : x \in \Omega_2\}$), then $ |\Omega_1| \geq  |\Omega_2|$. 

	\end{proof}
	
	The following lemma provides the desired upper bound for $ |\usupp| $ on the basis of Lemma \ref{lem:volumeofv}.
	%Here we define $ \Omegam := \{ x \in \Omega : \hat{u}(x)\leq m\} $ for some $m>0$.
	\begin{lem} \label{lem:dm}
		Given $ q \in [2,p^*+1) $ and $m>0$, if we have
		\begin{align}
			\label{ineq:uplus}
			\Lqnorm{\hat{u}_{+}}{\Omegam} > C_q \rho, 
		\end{align}
		then
		\begin{align}
			\label{ineq:volumem}
			|\usupp|\leq  |\Omegam|.			
		\end{align}
	\end{lem}
	
	\begin{proof}
		The enclosed solution $ u $ can be expressed by $ u=\hat{u}-\omega $, $\|\omega\|_{\Honezero}\leq \rho$.
		The embedding $ \Honezero \hookrightarrow \Lq$ confirms that $\|\omega\|_{\Lq}\leq C_q \rho$.
		We denote $c:=\|\omega\|_{\Lq}$,
		and then, have
		\begin{align*}
			|\usupp| \leq \max_{\Lqnorm{v}{\Omega}= c} |D(v)|.
		\end{align*}
		Lemma \ref{lem:volumeofv} ensures that this maximal value is realized when
		\begin{align*}
			v(x)
			=
			\left\{
			\begin{array}{ll}
				\hat{u}_+(x),&x\in D,\\
				0,&\text{\rm otherwise},
			\end{array}
			\right.		
		\end{align*}
		where $ D $ is a set that satisfies $\Lqnorm{\hat{u}_{+}}{D}= c$ and \eqref{eq:Dandm} for some $l \in \mathbb{R}$.
		The maximal value is $|D|$,  which is not greater than $|D(l)|$ due to \eqref{eq:Dandm}.
		Inequality \eqref{ineq:uplus} ensures that
		\begin{align}
			\label{eq:DandDm}
			\Lqnorm{\hat{u}_{+}}{D} = c \leq C_q \rho < \Lqnorm{\hat{u}_{+}}{D(m)}.
			%\Lqnorm{\hat{u}_{+}}{D} \leq \Lqnorm{\hat{u}_{+}}{D(m)}.
		\end{align}
		Suppose that $D(m) \subset D(l)$ in the strict sense.
		Then, we have $m<l$, and thus, $D(m) \subseteq \opD{l} \subseteq D$ due to \eqref{eq:Dandm}.
		This contradicts \eqref{eq:DandDm}.
		Therefore, we have $D(l) \subseteq D(m)$ and conclude \eqref{ineq:volumem}.
		%From \eqref{eq:DandDm}, we have $D(m)=\opD{l}=D$.

%		\\
%		\\
%		
%		
%		Let $D:=\max\{|D(v_c)|:\Lqnorm{v_c}{\Omega}= c\}$.
%		Lemma \ref{lem:volumeofv} ensures that the maximal value $D$ of $D(v)$ satisfies $\Lqnorm{\hat{u}_{+}}{D} = c$ and
%		$|\{ x \in \Omega : \hat{u}(x)< m\}| \leq D \leq |D(m)|$ for some $m \in \mathbb{R}$.

		%$c \leq C_q \rho$.
		%Suppose that $\|v\|_{\Lq}=c~(\leq C_q \rho)$.
		
		%Therefore, Lemma \ref{lem:volumeofv} ensures \eqref{ineq:volumem} because the maximal volume of $ \vDomain $, namely $\max\{|\vDomain| : v_c \in \Lq,~\|v_c\|_{\Lq}=c\}$, is a nondecreasing function with respect to $ c>0 $.
		
	\end{proof}
	
%	\begin{rem}
		The choice of $q$ does not greatly affect realizing \eqref{ineq:uplus}, and we can usually set $q=2$.
		Meanwhile, appropriately setting $m$ is important for confirming \eqref{ineq:uplus} as discussed below:
		Let $q$, $\rho$, and $\hat{u}$ be fixed.
		Then, $\Lqnorm{\hat{u}_{+}}{\Omegam}$ monotonically decreases as $m$ decreases, and $\Lqnorm{\hat{u}_{+}}{\Omegam} \downarrow 0$ as $m \downarrow 0$.
		Therefore, although smaller $m$ gives a good upper bound of $|\usupp|$ as in \eqref{ineq:volumem},
		too small $m>0$ leads to failure in ensuring \eqref{ineq:uplus}.
		Some concrete choices of $m$ can be found in our numerical experiments in Section \ref{sec:ex}.
%	\end{rem}
	\section{Lower bound for the minimal eigenvalue}
	\label{sec:lowerbound}
	The purpose of this section is to estimate a lower bound of the minimal eigenvalue $ \lambda_1(\usupp) $
	while we assume an $ H^1_0 $-estimation \eqref{eq:h10error} only; therefore, $ \usupp $ cannot be identified explicitly.
	To this end, we use the following Rayleigh–Faber–Krahn constant.
	
	\begin{theo}[\cite{faber1923bweis,krahn1925uber,krahn1926uber}]
	\label{theo:RFK}
		Inequality \eqref{ineq:Akn} with $k=1$ holds for
		\begin{align}
		\label{eq:A1n}
			A_{1,N} = B_N^{\frac{2}{N}} j_{\frac{N}{2}-1,1}^2
		\end{align}
		where $ B_N = \pi^{N/2}/\Gamma(N/2+1)$ denotes the volume of the unit $ N $-ball with the usual gamma function $\Gamma$, and $j_{\frac{N}{2}-1,1}$ is the first positive zero of the Bessel function of order $\frac{N}{2}-1$.
		The equality in \eqref{ineq:Akn} is attained if and only if $\Omega$ is a ball in $\mathbb{R}^{N}$.
	\end{theo}
	
	In general, evaluating $A_{1,N}$ in explicit decimal form using \eqref{eq:A1n} is not trivial.
	However, one can find in Table \ref{table:A1n} rigorous enclosures of $B_N$, $j_{\frac{N}{2}-1,1}$, and $A_{1,N}$ for several dimensions $N$.
	These were derived by strictly estimating all numerical errors; therefore, the correctness is mathematically guaranteed in the sense that correct values are included in the corresponding closed intervals.
	The enclosures were obtained using the kv library \cite{kashiwagikv}, a C++ based package for rigorous computations.
	The kv library includes four interval arithmetic operations and a function for rigorously calculating the gamma functions needed to derive $B_N$.
	However, no function for enclosing the Bessel function $j_{\frac{N}{2}-1,1}$ is built therein.
	Accordingly, we present a rigorous algorithm for calculating $j_{\frac{N}{2}-1,1}$ in
	Appendix \ref{appendix:rfkineq} based on the bisection method.
	\begin{table}[h]
	\caption{Strict enclosures of $B_N$, $j_{\frac{N}{2}-1,1}$, and $A_{1,N}$ for $ N=2,3,4,5 $.}
	\label{table:A1n}
	\begin{center}
		\renewcommand\arraystretch{1.3}
%		\footnotesize
		\begin{tabular}{llll}
			\hline
			$N$&$B_N$&$j_{\frac{N}{2}-1,1}$&$A_{1,N}$\\
			\hline
			\hline
			$2$&
			[3.1415926535, 3.1415926536]&
			[2.4048255576, 2.4048255577]&
			[18.1684145355, 18.1684145356] \\ 
			$3$&
			[4.1887902047, 4.1887902048]&
			[3.1415926535, 3.1415926536]&
			[25.6463452794, 25.6463452795] \\ 
			$4$&
			[4.9348022005, 4.9348022006]&
			[3.8317059702, 3.8317059703]&
			[32.6151384322, 32.6151384323] \\ 
			$5$&
			[5.2637890139, 5.2637890140]&
			[4.4934094579, 4.4934094580]&
			[39.2347942529, 39.2347942530] \\ 
			\hline
		\end{tabular}
	\end{center}
	\end{table}
	
	Combining Lemma \ref{lem:dm} and Theorem \ref{theo:RFK} for $ N=2,3 $ and using the lower bound in Table \ref{table:A1n}, we immediately have the following lower bounds for the minimal eigenvalue on $\usupp$.
	\begin{cor}
		\label{coro:lower_eigenm_n2n3}
		If \eqref{ineq:uplus} holds given $ q \in [2,p^*) $ and $m>0$, then we have
		\begin{align*}
		&\lambda_{1}(\usupp) \geq 18.1684145355 |\Omegam|^{-1},&N=2,\\
		&\lambda_{1}(\usupp) \geq 25.6463452794 |\Omegam|^{-\frac{2}{3}},&N=3.
		\end{align*}	
	\end{cor}
	
	\begin{rem}
	\label{rem:li}
	Instead of Theorem $\ref{theo:RFK}$, one can use the evaluation provided in {\rm \cite{li1983schrodinger}}:
	For a bounded domain $\Omega$, we have
	\begin{align}
		\lambda_{k}(\Omega) \geq \frac{4\pi^2 N}{N+2} \left(\frac{k}{B_N|\Omega|}\right)^{\frac{2}{N}}.
	\end{align}
	This estimation is somewhat rough compared with that in Corollary \ref{coro:lower_eigenm_n2n3} but stands alone in the sense that the lower bound can be calculated by hand as long as we know $B_N$.
	\end{rem}
	\section{Embedding constant}\label{sec:embedding}
	
	Explicitly estimating the embedding constant $C_p$ is important for our method.
	We use \cite[Corollary A.2]{tanaka2017sharp} to obtain an explicit value of $C_p$ for bounded domains based on the best constant in the classical Sobolev inequality provided in \cite{aubin1976,talenti1976}.\\
	\begin{theo}[{\cite[Corollary A.2]{tanaka2017sharp}}]
		\label{sitate}
		Let $\Omega \subset \mathbb{R}^{N}(N \geq 2)$ be a bounded domain, the measure of which is denoted by $| \Omega |$.
		Let $p \in (N/(N-1),2N/(N-2) ]$ if $N \geq 3$, $p \in (2, \infty)$ if $N=2$.
		Then, \eqref{embedding} holds for
		\begin{align*}
			C_p( \Omega ) = | \Omega |^{\frac{1}{N}+\frac{1}{p}-\frac{1}{2}}T_{p,N}.
		\end{align*}
		Here, $T_{p,N}$ is defined by
		\begin{align}
			\label{eq:tp}
			T_{p,N}=\pi^{-\frac{1}{2}} N^{-\frac{1}{q}}\left(\frac{q-1}{N-q}\right)^{1-\frac{1}{q}}\left\{\frac{\Gamma\left(1+\frac{N}{2}\right) \Gamma(N)}{\Gamma\left(\frac{N}{q}\right) \Gamma\left(1+N-\frac{N}{q}\right)}\right\}^{\frac{1}{N}},
		\end{align}
		where $\Gamma$ is the gamma function and $q=Np/(N+p)$.
	\end{theo}

	\begin{rem}
		Another formula to estimate the embedding constant $C_p$ can be found in {\rm \cite[Lemma 7.10]{nakaoplumwatanabe2019numerical}}, which is applicable not only to bounded domains but also to unbounded domains with a more generalized norm in $\Honezero$.
		Moreover, in {\rm\cite{tanaka2017sharp}}, one can find very sharp estimations of the best values of $C_p$ for $p=3,4,5,6,7$ on $\Omega=(0,1)^2$.
	\end{rem}

	Table \ref{table:embedding} shows the strict upper bounds of the embedding constants for several cases.
	These were evaluated using MATLAB 2019a with INTLAB version 11 \cite{rump1999book} with rounding errors strictly estimated;
	the required gamma functions were strictly computed via the function ``gamma'' packaged in INTLAB.
\begin{table}[h]
  \caption{Upper bounds for embedding constants.}
  \label{table:embedding}
    \begin{minipage}[t]{.45\textwidth}
    \begin{center}
    	\renewcommand\arraystretch{1.3}
		\begin{tabular}{lllll}
			\hline
			$N$& $|\Omega|$ &$p$& $T_{p,N}$ & $C_p(\Omega)$ \\
			\hline
			\hline
			2&1&3&$0.27991105$&$0.27991105$\\
			&&4&$0.31830989$&$0.31830989$\\
			&&5&$0.35780389$&$0.35780389$\\
			&&6&$0.39585400$&$0.39585400$\\[3pt]
			\hdashline
			&2&2&$0.28209480$&$0.56418959$\\
			&&3&$0.27991105$&$0.35266582$\\
			&&4&$0.31830989$&$0.37853639$\\
			&&5&$0.35780389$&$0.41100874$\\
			&&6&$0.39585400$&$0.44433110$\\
			\hline
		\end{tabular}
    \end{center}
  \end{minipage}
  \hfill
  \begin{minipage}[t]{.45\textwidth}
    \begin{center}
    	\renewcommand\arraystretch{1.3}
		\begin{tabular}{lllll}
			\hline
			$N$& $|\Omega|$ &$p$& $T_{p,N}$ & $C_p(\Omega)$ \\
			\hline
			\hline
			3&1&3&$0.26053089$&$0.26053089$\\
			&&4&$0.31800758$&$0.31800758$\\
			&&5&$0.37366932$&$0.37366932$\\
			&&6&$0.42726055$&$0.42726055$\\[3pt]
			\hdashline
			&2&2&$0.20540545$&$0.41730224$\\
			&&3&$0.26053089$&$0.29243603$\\
			&&4&$0.31800758$&$0.33691730$\\
			&&5&$0.37366932$&$0.38240342$\\
			&&6&$0.42726055$&$0.42726055$\\
			\hline
		\end{tabular}
    \end{center}
  \end{minipage}
\end{table}
	%\end{rem}

	In the case of $p=2$, to which Theorem \ref{sitate} is inapplicable, the following best evaluation can be used instead of Theorem \ref{sitate}:
	\begin{align}
		\label{poincarebest}
		\|u\|_{L^{2}(\Omega)} \leq \frac{1}{\sqrt{\lambda_{1}(\Omega)}}\|u\|_{\Honezero}.
	\end{align}
	For example, when $\Omega=(0,1)^N$ $(N=1,2,3,\cdots)$, we have the exact eigenvalue $\lambda_{1}(\Omega)=N\pi^{2}$.                                            
	Even otherwise, lower bounds of $\lambda_{1}(\Omega)$ (and therefore upper bounds for the embedding constant) can be estimated for bounded domain $\Omega$ using the formulas in Section \ref{sec:lowerbound}.
	\section{A posteriori verification for positivity}
	\label{sec:positive}
	In this section, we design a unified method for proving positivity.
	We first clarify the assumption imposed on nonlinearity $f$ with some explicit parameters.
	Let $ f $ satisfy
	\begin{align}
	\label{f:theo1}
	-f(-t)\leq \lambda t + \displaystyle \sum_{i=1}^{n}a_{i}t^{p_{i}} \text{~~for~all~~}t\geq 0
	\end{align}
	for some $ \lambda \in \mathbb{R} $, nonnegative coefficients $ a_1,a_2,\cdots,a_n $, and subcritical exponents $ p_1,p_2,\cdots,p_n \in (1,p^*)$.
	This assumption does not break the generality of $f$ satisfying \eqref{asuume:f}.
	Our algorithm for verifying positivity is based on the following argument, a generalization of \cite[Thorem 2.1]{tanaka2020numerical}.
	In the following, $\lambda_1(\usupp)$ and $C_{p+1}(\usupp)$ are interpreted as $\lambda_1(\text{int\,}(\usupp))$ and $C_{p+1}(\text{int\,}(\usupp))$, respectively, where $\text{int\,}(\usupp)$ denotes the interior of $\usupp$.
	We denote by $\lowerboundoflambda \in \mathbb{R}$ a lower bound of $\lambda_{1}(\usupp)$.
%	When the interior of $\usupp$ is empty, we can set $\lowerboundoflambda$ and $C_{p+1}(\usupp)$ any real numbers.
	When the interior of $\usupp$ is empty, we can set $\lowerboundoflambda$ to an arbitrarily large value.
	
	\begin{lem}\label{theo1}
		Suppose that a solution $u \in \Honezero$ of \eqref{main:fpro} exists in $ \ol{B}(\hat{u},\rho) $, given
		some approximation $ \hat{u} \in \Honezero$ and radius $\rho>0$.
		If	
		\begin{align}
%		&\lambda < \lambda_1(\usupp)
		&\lambda < \lowerboundoflambda
		\shortintertext{and}
%		&\displaystyle \sum_{i=1}^{n}a_i C_{p_{i}+1}(\usupp)^{2}\left( \left\|\hat{u}_{-}\right\|_{L^{p_{i}+1}}+C_{p_{i}+1}\rho\right)^{p_{i}-1}<1 - \f{\lambda}{\lambda_1(\usupp)},\label{cond:theo1}
		&\displaystyle \sum_{i=1}^{n}a_i C_{p_{i}+1}(\usupp)^{2}\left( \left\|\hat{u}_{-}\right\|_{L^{p_{i}+1}}+C_{p_{i}+1}\rho\right)^{p_{i}-1}<1 - \f{\lambda}{\lowerboundoflambda},\label{cond:theo1}
		\end{align}	
		then $u$ is nonnegative.
		Note that when $ \usupp $ is disconnected, \eqref{cond:theo1} is understood as the set of inequalities for all connected components of $ \usupp $.
	\end{lem}
	\begin{proof}
		We extend the proof of \cite[Thorem 2.1]{tanaka2020numerical} to achieve a more precise evaluation.
		For $p\in (1,p^*)$, we have
		\begin{align}
		\label{uandu_}
		\left\|u_{-}\right\|_{L^{p+1}}
		\leq
		\left\|\hat{u}_{-}\right\|_{L^{p+1}}+ C_{p+1}\rho;
		\end{align}
		see the proof of \cite[Thorem 2.1]{tanaka2020numerical}.
		We then prove that the norm of $ u_{-} $ vanishes.
		Because $u$ satisfies
		\begin{align*}
		\left(\nabla u,\nabla v\right)_{L^2}=\left<F(u),v\right>{\rm~~for~all~}v\in \Honezero,
		\end{align*}
		by fixing $v=u_{-}$, we have from \eqref{f:theo1} that
		\begin{align}
		\left\|u_{-}\right\|_{\Honezero}^{2}
		\leq &\displaystyle \int_{\Omega}\left\{\lambda\left(u_{-}(x)\right)^{2}+\sum_{i=1}^{n}a_{i}\left(u_{-}(x)\right)^{p_{i}+1}\right\}dx \nonumber \\
		= &\displaystyle \lambda\left\|u_{-}\right\|_{L^{2}}^{2}+\sum_{i=1}^{n}a_{i}\left\|u_{-}\right\|_{L^{p_{i}+1}}^{p_{i}+1}\nonumber \\
		\leq &\left\{ \displaystyle \frac{\lambda}{\lowerboundoflambda} +\sum_{i=1}^{n}a_{i}C_{p_{i}+1}(\usupp)^{2} \left\|u_{-}\right\|_{L^{p_{i}+1}}^{p_{i}-1} \right\} \left\|u_{-}\right\|_{\Honezero}^{2}.
		\end{align}
		Inequalities \eqref{cond:theo1} and \eqref{uandu_} lead to
		\begin{align}
			\displaystyle \frac{\lambda}{\lowerboundoflambda} +\sum_{i=1}^{n}a_{i}C_{p_{i}+1}(\usupp)^{2} \left\|u_{-}\right\|_{L^{p_{i}+1}}^{p_{i}-1} < 1,
		\end{align}
		which ensures $ \|u_{-}\|_{\Honezero} = 0 $.
		Thus, the nonnegativity of $ u $ is proved.
		
	\end{proof}
	\begin{rem}
		The maximum principle ensures the positivity of $u$ (i.e., $u(x)>0$ inside $\Omega$) from its nonnegativity for a wide class of nonlinearities $f$.
		See, for example, {\rm\cite{drabek2009maximum}} for a generalized maximum principle applicable for weak solutions.
	\end{rem}	
	%\begin{rem}
	
	\begin{rem}\label{rem:ifrhosmall}
%		The formula $ \left\|\hat{u}_{-}\right\|_{L^{p_{i}+1}}+C_{p_{i}+1}\rho$ in parentheses in \eqref{cond:theo1} goes to $0$ as $\hat{u}$ approaches a nonnegative function and $ \rho \downarrow 0 $.
%		Therefore, as long as verification succeeds for a nonnegative approximation $\hat{u}$ with sufficient accuracy,
%		the nonnegativity of $ u $ can be confirmed using Lemma $\ref{theo1}$.		
		The formula $ \left\|\hat{u}_{-}\right\|_{L^{p_{i}+1}}+C_{p_{i}+1}\rho$ in parentheses in \eqref{cond:theo1} indicates that the nonnegativity of $ u $ can be confirmed using Lemma $\ref{theo1}$
		if the norm of $\hat{u}_{-}$ is sufficiently small and verification succeeds with sufficient accuracy.
		 
	\end{rem}
	
	\begin{rem}
		Lemma $\ref{theo1}$ can be indeed regarded as a generalization of {\rm \cite[Theorem 2.1 and Corollary A.1]{tanaka2020numerical}} because Lemma $\ref{theo1}$ is formally obtained from {\rm \cite[Thorem 2.1]{tanaka2020numerical}} via the replacements $\lambda_1(\Omega) \rightarrow \lowerboundoflambda$ and the left-side $C_{p_{i}+1}(\Omega) \rightarrow C_{p_{i}+1}(\usupp)$.
		Actually, the replacement for $\lambda_1(\Omega)$ was already discussed in {\rm \cite[Corollary A.1]{tanaka2020numerical}}.
		However, the embedding constant $C_{p_{i}+1}(\Omega)$ was not replaced.
		Because the formula $ \left\|\hat{u}_{-}\right\|_{L^{p_{i}+1}}+C_{p_{i}+1}\rho$ in parentheses in \eqref{cond:theo1} can be very small as mentioned in Remark \ref{rem:ifrhosmall},
		replacing $C_{p_{i}+1}(\usupp)$ with its rough bound $C_{p_{i}+1}(=C_{p_{i}+1}(\Omega))$
		is rarely problematic.
		However, in several cases, roughly estimating the lower bound for $\lambda_1(\usupp)$ via $\lambda_1(\usupp) \geq \lambda_1(\Omega)$ cannot satisfy \eqref{cond:theo1}.
		For example, when $f(t)=\lambda (t - t^3)$, $\lambda$ should satisfy $\lambda \geq \lambda_1(\Omega)$ to admit a positive solution.
	\end{rem}
		
	By applying the estimations in Sections \ref{sec:lowerbound} and \ref{sec:embedding} to Lemma \ref{theo1}, we have the following theorem. The constants in this theorem can be computed explicitly using formulas provided before.
	
	\begin{theo}\label{theo:main}
		Suppose that a solution $u \in \Honezero$ of \eqref{main:fpro} exists in $ \ol{B}(\hat{u},\rho) $, given
		some approximation $ \hat{u} \in \Honezero$ and radius $\rho>0$.
		Moreover, we assume \eqref{ineq:uplus}, given $ q \in [2,p^*+1) $ and $m>0$.
		Then, we define $\mathcal C_1=\mathcal C_1(N,f,\hat{u},\rho,m)$ and $\mathcal C_2=\mathcal C_2(N,f,\hat{u},m)$ as
		\begin{align*}
			\mathcal C_1 &:=\displaystyle \sum_{i=1}^{n}a_i | D(m) |^{\frac{2}{N}+\frac{2}{p_i+1}-1}T_{i}^2
			\left( \left\|\hat{u}_{-}\right\|_{L^{p_{i}+1}}+| \Omega |^{\frac{1}{N}+\frac{1}{p_i+1}-\frac{1}{2}}T_{i}\rho\right)^{p_{i}-1},\\
			\mathcal C_2 &:= 1 - \f{\lambda}{A_{1,N}} |D(m)|^{\frac{2}{N}},
		\end{align*}
		where $T_{i} := T_{p_i+1,N}$ is defined by \eqref{eq:tp} and $A_{1,N}$ is the constant of \eqref{eq:A1n}.
		If $\mathcal C_1 < \mathcal C_2$, $u$ is nonnegative.
	\end{theo}
	
	\begin{table}[t]
	\caption{Calculation examples of $\mathcal C_1$ and $\mathcal C_2$ for several nonlinearities $f$.}
	\label{table:concrete_nonlinearities}
	\begin{center}
		\renewcommand\arraystretch{1.8}
		%\footnotesize
		\begin{tabular}{lll}
			\hline
			$N$&$\mathcal C_1$&$\mathcal C_2$\\
			\hline
			\hline
			
			\multicolumn{3}{l}{$f(t)=\lambda t + |t|^{p-1}t$ with $p \in (1, p^*)$}\\
			$2$&
			$|D(m)|^{\frac{2}{p+1}} T_{p+1,2}^2 \left( \left\|\hat{u}_{-}\right\|_{L^{p+1}}+|\Omega|^{\frac{1}{p+1}} T_{p+1,2}\rho\right)^{p-1}$&
			$1 - \frac{\lambda}{T_{p+1,2}} |D(m)|$ \\ 
			$3$&
			$|D(m)|^{\frac{2}{p+1}-\frac{1}{3}} T_{p+1,3}^2\left( \left\|\hat{u}_{-}\right\|_{L^{p+1}}+|\Omega|^{\frac{1}{p+1}-\frac{1}{6}} T_{p+1,3}\rho\right)^{p-1}$&
			$1 - \frac{\lambda}{T_{p+1,3}} |D(m)|^\frac{2}{3}$ \\
			
			\hline
			\multicolumn{3}{l}{$f(t)=\lambda (t - t^3)$ where $-f(-t) \leq \lambda t$ for all $t \geq 0$}\\
			$2$&
			$0$&
			$1 - \frac{\lambda}{T_{p+1,2}} |D(m)|$ \\ 
			$3$&
			$0$&
			$1 - \frac{\lambda}{T_{p+1,3}} |D(m)|^\frac{2}{3}$ \\
			
			\hline
			\multicolumn{3}{l}{$f(t)=\lambda (-at+(1+a) |t|t - t^3) $ where $-f(-t) \leq \lambda (1+a) |t|t$ for all $t \geq 0$}\\
			$2$&
			$ \lambda (1+a) |D(m)|^{\frac{2}{3}} T_{3,2}^2 \left( \left\|\hat{u}_{-}\right\|_{L^{3}}+T_{3,2}|\Omega|^{\frac{1}{3}}\rho\right)$&
			$1$ \\ 
			$3$&
			$ \lambda (1+a) |D(m)|^{\frac{1}{3}} T_{3,3}^2 \left( \left\|\hat{u}_{-}\right\|_{L^{3}}+T_{3,3}|\Omega|^{\frac{1}{6}} \rho\right)$&
			$1$ \\

			\hline
			\multicolumn{3}{l}{$ f(t)=\lambda (t+\lionsA |t|t - \lionsB t^3) $ where $-f(-t) \leq \lambda (t+\lionsA |t|t)$ for all $t \geq 0$}\\
			$2$&
			$ \lambda \lionsA |D(m)|^{\frac{2}{3}} T_{3,2}^2 \left( \left\|\hat{u}_{-}\right\|_{L^{3}}+T_{3,2}|\Omega|^{\frac{1}{3}}\rho\right)$&
			$1 - \frac{\lambda}{T_{p+1,2}} |D(m)|$ \\ 
			$3$&
			$ \lambda \lionsA |D(m)|^{\frac{1}{3}} T_{3,3}^2 \left( \left\|\hat{u}_{-}\right\|_{L^{3}}+T_{3,3}|\Omega|^{\frac{1}{6}} \rho\right)$&
			$1 - \frac{\lambda}{T_{p+1,3}} |D(m)|^\frac{2}{3}$ \\
			\hline
		\end{tabular}
	\end{center}
	\end{table}	
	Table \ref{table:concrete_nonlinearities} shows calculation examples of $\mathcal C_1$ and $\mathcal C_2$ for some concrete nonlinearities $f$,
	where we can use the estimations of $T_{p,N}$ in Tables \ref{table:A1n} and \ref{table:embedding}.
	For the first nonlinearity $f(t)=\lambda t + |t|^{p-1}t$ with a subcritical exponent $p \in (1, p^*)$,
	positive solutions are admitted only when $\lambda<\lambda_{1}(\Omega)$.
	Therefore, calculating $|D(m)|$ can be avoided if we can evaluate $\lambda_{1}(\Omega)$ with sufficient accuracy. 
	Note that $\lambda_{1}(\Omega)$ can be obtained analytically, such as when $\Omega$ is a hyperrectangle.
	When $\lambda=0$, $|D(m)|$ does not need to be calculated to derive $\mathcal C_2$.
	In cases where we do not require $|D(m)|$ to estimate $\mathcal C_2$, it is reasonable to replace $|D(m)|$ with $|\Omega|$ in calculating $\mathcal C_1$ to reduce calculation cost especially when $\rho$ is sufficiently small.
	%because this replacement rarely causes a problem when $\rho$ is sufficiently small.
	%, whereas $\mathcal C_1$ can become somewhat rough.
	
	The values of the gamma functions should be estimated explicitly to compute $T_{p+1,N}$.
	There are several toolboxes that enable us to calculate the gamma functions rigorously.
	Indeed, as mentioned in Sections \ref{sec:lowerbound} and \ref{sec:embedding}, kv library \cite{kashiwagikv} and INTLAB \cite{rump1999book} have such a rigorous function.
	Therefore, we are left to estimate upper bounds for $\left\|\hat{u}_{-}\right\|_{L^p}$ and $|D(m)|$, and a lower bound for $\Lpnorm{\hat{u}_{+}}{\Omegam}$.
	We describe the ways to calculate these bounds in the following subsections.
	For sufficient accuracy, we divide $\overline{\Omega}$ into a union of small subsets $\left\{K_{i}\right\}_{i=1}^{N_K}$ that satisfies $\cup_{i}\overline{K_{i}}=\overline{\Omega}$ and ${\rm measure\,} (\overline{K_{i}} \cap \overline{K_{j}})=0$ for $i\neq j$.
	We can obtain such a division ``for free'' such as when using finite element methods to compute $\hat{u}$.
	Otherwise, we should create a mesh $\left\{K_{i}\right\}_{i=1}^{N_K}$ of $\overline{\Omega}$ that satisfies the above property.
	When $\Omega$ is polygonal, a convenient way is to divide $\overline{\Omega}$ into a rectangular or triangular mesh.
	In preparation, we estimate a lower bound $m_i$ and an upper bound $M_i$ for $\hat{u}$ on each mesh $\overline{K_{i}}$, obtaining a closed interval $[m_i, M_i]$ that encloses $[\displaystyle \min_{x \in \overline{K_i}} \hat{u}(x), \displaystyle \max_{x \in \overline{K_i}} \hat{u}(x)]$.
	When we use a linear finite element approximation, $m_i$ and $M_i$ are the minimal and maximal values at vertexes for each $K_i$, respectively.
	We denote by $\overline{\Lambda}_m$ and $\underline{\Lambda}_m$ the sets of indices $i$ such that $M_i \leq m$ and $m_i \leq m$, respectively.

	\subsection{Upper bound of $\left\|\hat{u}_{-}\right\|_{L^p}$ for $p \in (1,\infty)$}
	\label{subsec:upperbounduminus}
	To calculate an upper bound of $\left\|\hat{u}_{-}\right\|_{L^p}$,
	we use the inequality
	\begin{align}
	\label{ineq:Lpuminus}
	\left\|\hat{u}_{-}\right\|_{L^{p}(\Omega)}
	= \left(\int_{\usupp} \hat{u}_{-}(x)^{p}dx\right)^{\frac{1}{p}}
	\displaystyle \leq \left( \max_{x\in \usupp} \hat{u}_{-}(x) \right) \left|\usupp\right|^{\frac{1}{p}}.
	\end{align}
	Upper bounds for $\displaystyle \max_{x\in \usupp}\hat{u}_{-}(x)$ and $\left|\usupp\right|$ can be obtained as follows:
	\begin{align*}
		&\displaystyle \max_{x\in \usupp} \hat{u}_{-}(x) \leq \displaystyle | \min_{i} [ \min \left\{ 0,m_i \right\} ] |,\\
		&\left|\usupp\right| \leq \displaystyle \sum_{m_i < 0} \left| K_i \right|.
	\end{align*}
	Thus, we can obtain the desired lower bound according to \eqref{ineq:Lpuminus}.

	%The following steps evaluate upper bounds for $\displaystyle \max_{x\in \usupp}\hat{u}_{-}(x)$ and $\left|\usupp\right|$.
%	\begin{enumerate}
	%\item  Compute $\displaystyle \min_{x \in K_i} \hat{u}(x)$ or its lower bound for each $i$, and set this to $m_i$. 
	%Note that the notation $m_i$ is different from $m$, a positive number used in the definition of $D(m)$.
	%For example when we use a linear finite element, $m_i$ is the minimal value at vertexes for each $K_i$.
%	\item Compute the magnitude $\displaystyle | \min_{i} [ \min \left\{ 0,m_i \right\} ] |$, which is an upper bound for $\displaystyle \max_{x\in \usupp} \hat{u}_{-}(x) $.
%	\item Compute the sum of volumes $\displaystyle \sum_{m_i < 0} \left| K_i \right|$, which is an upper bound for $\left|\usupp\right|$.
%	\end{enumerate}

	\subsection{Upper bound of $|D(m)|$}
	\label{subsec:dm}
	Recall that $D(m):= \{ x \in \Omega : \hat{u}(x)\leq m\}$.
	When $D(m) \cap \overline{K_i} \neq \phi$, we see that $m_i \leq m$.
	%For each $i \in \underline{\Lambda}_m$, $m_i \leq m$.
	Therefore, it follows from the definition of $\underline{\Lambda}_m$ that
	\begin{align*}
		D(m)\subset \bigcup_{i\in \underline{\Lambda}_m}K_{i}
%	\end{align*}
	\text{~~~and~~~}
	%An upper bound for the volume of $D(m):= \{ x \in \Omega : \hat{u}(x)\leq m\}$ is computed as follows
%	\begin{align*}
		|D(m)| \leq \displaystyle \sum_{i \in \underline{\Lambda}_m} |K_i|.
	\end{align*}
%	because 
	
%	We calculate an upper bound for the volume of $D(m):= \{ x \in \Omega : \hat{u}(x)\leq m\}$ given $m>0$ via the following steps.
	
	%$D(m)$\\
%	$\displaystyle \bigcup_{i\in \overline{\Lambda}_m}K_{i}\subset D(m)\subset \bigcup_{i\in \underline{\Lambda}_m}K_{i}$

%	Therefore, it follows that
%	\begin{align*}
%		|D(m)| \leq \displaystyle \sum_{i \in \underline{\Lambda}_m} |K_i|.
%	\end{align*}
	
%	\begin{enumerate}
%		\item Compute $\displaystyle \max_{x \in K_i} \hat{u}(x)$ or its upper bound for each $i$, and set this to $M_i$.
%		When we use a linear finite element, $M_i$ is the maximal value at vertexes for each $K_i$.\\
%		\item Let $\overline{\Lambda}_m$ be the set of indices $i$ such that $M_i \leq m$.
%			Compute the sum of volumes $\displaystyle \sum_{i \in \overline{\Lambda}_m} |K_i|$,
%			which is an upper bound for $|D(m)|$.
%	\end{enumerate}
	
	\subsection{Lower bound of {$\Lpnorm{\hat{u}_{+}}{\Omegam}$ for $p \in (1,\infty)$}}
	It should be noted that $\Lpnorm{\hat{u}_{+}}{\Omegam}$ needs to be estimated from below,
	whereas upper bounds are required for the other constants.
	More precise estimation is required to satisfy \eqref{ineq:uplus} compared with $\left\|\hat{u}_{-}\right\|_{L^p}$.
	%Let $\underline{\Lambda}_m$ be the set of indices $i$ such that $m_i \leq m$, so that 
	Therefore, we use the following estimation:
	\begin{align}
		&\Lpnorm{\hat{u}_{+}}{\Omegam}
		= \left(\int_{D(m)} \hat{u}_{+}(x)^{p}dx\right)^{\frac{1}{p}}\nonumber \\
		&\geq \left(\sum_{i\in\overline{\Lambda}_{m}}\int_{K_{i}} \hat{u}_{+}(x)^{p}dx\right)^{\frac{1}{p}}
		\geq \left(\sum_{i\in\overline{\Lambda}_{m}}|K_{i}|\left(\min_{x\in K_{i}}\hat{u}_{+}(x)\right)^{p}\right)^{\frac{1}{p}},
	\end{align}
	where $\displaystyle \bigcup_{i\in \overline{\Lambda}_m}K_{i}\subset D(m)$.
	%This estimation is more precise than in \eqref{ineq:Lpuminus}.
	By applying $\displaystyle \min_{x\in K_{i}}\hat{u}_{+}(x) \geq \max\{0,m_i\}$ to each $i \in \overline{\Lambda}_m$, estimation of $\Lpnorm{\hat{u}_{+}}{\Omegam}$ from below is completed.
%	We are left to estimate $\displaystyle \min_{x\in K_{i}}\hat{u}_{+}(x)$ from below for $i \in \overline{\Lambda}_m$.
%	This actually coincides with $\displaystyle m^+_i=\max\{0,m_i\}$.
%	Because $m_i$ or its lower bound is already calculated in Subsection \ref{subsec:upperbounduminus} for each $i$,
%	estimation of $\Lpnorm{\hat{u}_{+}}{\Omegam}$ from below is completed.
	%\section{Application to enclosed solutions}\label{sec:ex}

	%we present examples in which the positivity of solutions of \eqref{main:fpro} on a square domain $ \Omega=(0,1)^2 $ and an L-shaped domain $ \Omega=(0,1)^2\backslash([0,0.5]\times[0.5,1]) $ are verified via the proposed method.
	
	\section{Verification theory}
	In this section, we prepare verification theory to prove the existence of solutions $u$ of \eqref{main:fpro} satisfying \eqref{eq:h10error} and apply our method to verifying the positivity of $u$.
	%for numerically-computed approximation $\hat{u}$.
	%We apply our method to verifying the positivity of solutions $u$ of \eqref{main:fpro} that are proved to exist nearby its approximation $\hat{u}$ and satisfy \eqref{eq:h10error}.
	We consider a square domain $ \Omega=(0,1)^2 $, and an L-shaped domain $ \Omega=(0,1)^2\backslash([0,0.5]\times[0.5,1]) $
	where $H^2$-regularity of solutions is lost due to the re-entrant corner at $(0.5,0.5)$.
	To obtain $H^1_0$-estimation \eqref{eq:h10error}, we use the following affine invariant Newton–Kantorovitch theorem. 
	We omit the notation $(\Omega)$ when expressing operator norms just to save space.
	For example, we abbreviate $||\cdot||_{{\cal L}(\Honezero,\Honezero)}$ as $||\cdot||_{{\cal L}(H^1_0,H^1_0)}$.
	We denote $B(\hat{u},r):= \{v\in \Honezero : \left\|v-\hat{u}\right\|_{H^1_0}<r\}$ and $\overline{B}(\hat{u},r):= \{v\in \Honezero : \left\|v-\hat{u}\right\|_{H^1_0}\leq r\}$ for $r>0$.
	
\begin{theo}[\cite{deuflhard1979affine}]
	\label{theo:nk}
	Let $ \hat{u} \in \Honezero$ be some approximate solution of $\mc F(u)=0$.
	Suppose that there exists some $\alpha>0$ satisfying
	\begin{align}
		\label{alpha}
		||{\mc F'_{\hat{u}}}^{-1}\mc F(\hat{u})||_{H^1_0}\le\alpha.
	\end{align}
	Moreover, suppose that there exists some $\beta>0$ satisfying
	\begin{align}
		\label{beta}
		||{\mc  F'_{\hat{u}}}^{-1}(\mc  F'_v-\mc  F'_w)||_{{\cal L}(H^1_0,H^1_0)}\le\beta||v-w||_{H^1_0}~~\mathrm{for~all~} v,w \in D,
	\end{align}
	where $D=B(\hat{u},2\alpha+\delta)$ is an open ball depending on the above value $\alpha>0$ for small $\delta>0$.
	If
	\begin{align}
		\label{alphabeta}
		\alpha\beta \le \frac{1}{2},
	\end{align}
	then there exists a solution $u \in \Honezero$ of $\mc F(u)=0$ in $\ball$ with
	\begin{align}
		\label{nkerror}
		\rho = \frac{1-\sqrt{1-2\alpha\beta}}{\beta}.
	\end{align}
	Furthermore, $ \mc  F'_{\varphi} $ is invertible for every $ \varphi \in \oballrho$, and the solution $u$ is unique in $\overline{B}(\hat{u},2\alpha)$.
\end{theo}
%Note that our positivity-verifying method limits neither the basis functions that constitute approximate solutions nor the verification theory for obtaining $H^1_0$-estimation,
%but can be applied to various bases such as the linear finite element basis or the Fourier basis with any verification methods for computing $H^1_0$-errors (again, see \cite{nakaoplumwatanabe2019numerical}).

We set $ \alpha $ and $ \beta $ to upper bounds of
\begin{align}
	\label{ineq:alphabeta}
	\|\mathcal{F}_{\hat{u}}^{\prime-1}\|_{\mathcal{L}\left(H^{-1}, H^1_0\right)}\|\mathcal{F}(\hat{u})\|_{H^{-1}} \text{~~~and~~~} 
	\|\mathcal{F}_{\hat{u}}^{\prime-1}\|_{\mathcal{L}\left(H^{-1}, H^1_0\right)} L,
\end{align}
respectively, then applying Theorem \ref{theo:nk} to prove the local existence of solutions.
Here, $ L $ is a positive number satisfying
\begin{align}
	\label{ineq:lip}
	\left\|F_{v}^{\prime}-F_{w}^{\prime}\right\|_{\mathcal{L}(H^1_0, H^{-1})} \leq L\|v-w\|_{H^1_0} \text {~~for all~~} v, w \in D.
\end{align}	
	We estimate the inverse norm $ \|\mathcal{F}_{\hat{u}}^{\prime-1}\|_{\mathcal{L}\left(H^{-1}, H^1_0\right)} $ using the method described in \cite{tanaka2014verified,liu2015framework} in a finite-dimensional subspace $ V_M \subset \Honezero$ specified later.

	\subsection{Square domain}
	For the square domain $\Omega = (0,1)^2$,
	we construct approximate solutions $\hat{u}$ with a Legendre polynomial basis.
	More precisely,	we construct $\hat{u}$ as
	\begin{align}
	\label{eq:Legendrebase}
	\displaystyle \hat{u}(x,y)=\sum_{i=1}^{M}\sum_{j=1}^{M}u_{i,j}\phi_{i}(x)\phi_{j}(y),~~u_{i,j}\in \mathbb{R},
	\end{align}
	where each $\phi_{n}$ ($ n=1,2,3,\cdots $) is defined by
	\begin{align}
	\displaystyle &\phi_{n}(x)=\frac{1}{n(n+1)}x(1-x)\frac{dQ_{n}}{dx}(x)\nonumber \\
	&\text{~~with~~}
	Q_{n}(x)=\displaystyle \frac{(-1)^{n}}{n!}\left(\frac{d}{dx}\right)^{n}x^{n}(1-x)^{n},~~n=1,2,3,\cdots.
	\end{align}
	The upper bound on $ \|\mathcal{F}(\hat{u})\|_{H^{-1}} $ is evaluated via $C_2 \|\mathcal{F}(\hat{u})\|_{L^2} $,
	where $C_2$ is the constant of embedding $ \Ltwo \hookrightarrow H^{-1} $, which in fact coincides with the constant of embedding $ \Honezero \hookrightarrow \Ltwo $ (see, for example, \cite{plum2008}).	
	This $ L^2$-norm is computed using a numerical integration method with a strict estimation of rounding errors using \cite{kashiwagikv}.
	For $\Omega = (0,1)^2$, as mentioned for \eqref{poincarebest}, the embedding constant $ C_2 $ is calculated as $ C_2=(2 \pi^2)^{-\f{1}{2}} \approx 0.2251$ with a strict estimation of rounding errors.
	
	We define a finite-dimensional subspace $ V_M~(\subset \Honezero) $ as the tensor product $V_M = \text{span\,}\{\phi_{1},\phi_{2},\cdots,\phi_{M}\} \otimes \text{span\,}\{\phi_{1},\phi_{2},\cdots,\phi_{M}\} $, then define the orthogonal projection $ P_M $ from $ \Honezero $ to $ V_M $ as
		\begin{align}
			\label{eq:PM}
			(v- P_M v,v_M)_{H^1_0} = 0 \text{~~~for all~} v \in \Honezero \text{~and~} v_M \in V_M.
		\end{align} 	
		We use \cite[Theorem 2.3]{kimura1999on} to obtain an explicit interpolation-error constant $ C_M $ satisfying
		\begin{align}
			\label{eq:interpolation}
			\left\|v-P_{M}v\right\|_{H^1_0}\leq C_{M}\left\|\Delta v\right\|_{L^{2}} \text{~~for~all~~} v \in \Honezero \cap H^2(\Omega),
		\end{align}
		then applying \cite{tanaka2014verified,liu2015framework} to estimate the inverse norm $ \|\mathcal{F}_{\hat{u}}^{\prime-1}\|_{\mathcal{L}\left(H^{-1}, H^1_0\right)} $.
	%In our examples,
	%the inverse norm $ \|\mathcal{F}_{\hat{u}}^{\prime-1}\|_{\mathcal{L}\left(H^{-1}, H^1_0\right)} $ was estimated using the method described in \cite{tanaka2014verified,liu2015framework} in the finite-dimensional subspace $ V_M $.

	\subsection{L-shaped domain}	
	For the L-shaped domain $ \Omega=(0,1)^2\backslash([0,0.5]\times[0.5,1]) $, 
	we set $V_M$ to a finite element space of piecewise quadratic basis functions with the non-uniform triangulation displayed in Fig.~\ref{fig:Lshape},
	constructing approximate solutions $\hat{u} \in V_M$.
	%$V_M$ is 
	Using \cite[Theorem 3.3]{liu2013verified}, we confirmed that $C_M = 0.011437$ satisfies
	\begin{align}
		\label{CNweek}
		\left\|u_{g}-P_{M}u_{g}\right\|_{H^1_0}\leq C_{M}\left\|g\right\|_{L^{2}}~~~{\rm for~all}~g\in L^2 (\Omega),
	\end{align}
	where $ P_M $ is the orthogonal projection from $ \Honezero $ to $ V_M $ defined in \eqref{eq:PM},
	and $u_g\in {H^1_0(\Omega )}$ is a unique solution of the weak formulation of the Poisson equation
	\begin{align}
		\label{poisson}
		\left(u_g, v\right)_{H^1_0}=\left(g,v\right)_{L^2}~~~{\rm for~all}~v\in H^1_0 (\Omega)
	\end{align}
	given $g\in L^{2}\left(\Omega\right)$.
	The upper bound on $ \|\mathcal{F}(\hat{u})\|_{H^{-1}} $ is evaluated using the Raviart-Thomas mixed finite element method (see, for example, \cite{takayasu2013verified}).
	We estimate the inverse norm $ \|\mathcal{F}_{\hat{u}}^{\prime-1}\|_{\mathcal{L}\left(H^{-1}, H^1_0\right)} $ using \eqref{CNweek} and the method described in \cite{tanaka2014verified,liu2015framework}.
	
	%This equation has a unique solution $u_g\in {H^1_0(\Omega )}$ for each $g\in L^{2}\left(\Omega\right)$ \cite{grisvard2011elliptic}.
	%which depends on $\tau$ by the inner product \eqref{Vnorm} endowed with $H^1_0(\Omega)$.
	%The following theorem enables us to estimate the error between $\lambda_{k}$ and $\lambda_{k}^{M}$.
	%	Let $\hat{u} \in H^1_0(\Omega) \cap L^{\infty}(\Omega)$.
	%	Suppose that there exists $C^{\tau}_{M}>0$ such that
	%for any $g\in L^{2}\left(\Omega\right)$ and the corresponding solution $u_{g}\in {H^1_0(\Omega )}$ of \eqref{poisson}.
	%More detailed discussion can be found in
	%\cite{takayasu2013verified}.
	\begin{figure}[h]
			\begin{center}
				\includegraphics[height=65 mm]{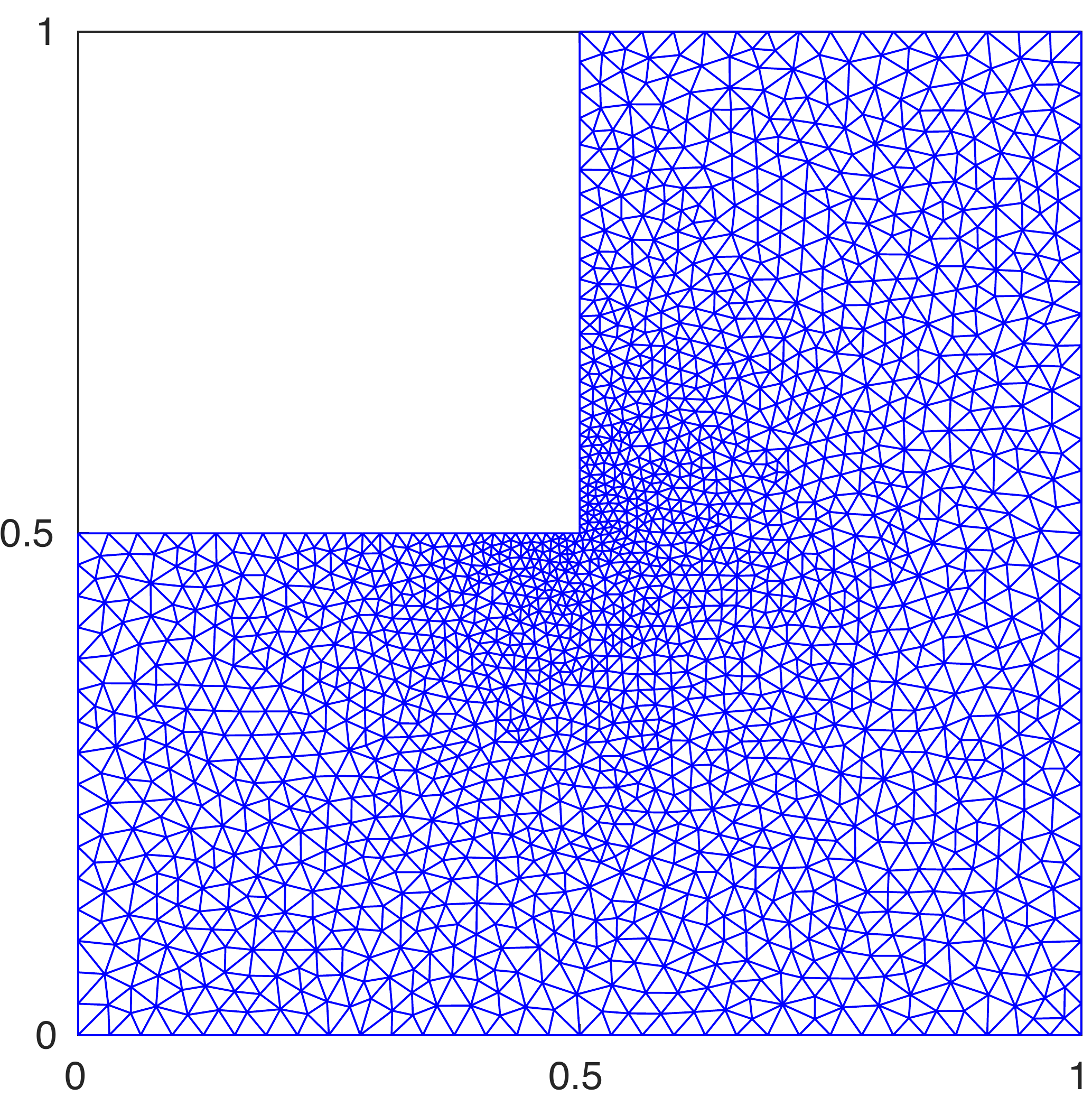}
			\end{center}
		\caption{A non-uniform mesh for the L-shaped domain $\Omega=(0,1)^2\backslash([0,0.5]\times[0.5,1])$.
		The number of triangular elements is 3754.
		The mesh size around the non-convex corner is about a quarter of the other corners.
		}
		\label{fig:Lshape}
	\end{figure}

	\begin{rem}
	Solutions of \eqref{main:fpro} on the L-shaped domain may lose $H^2$-regularity due to the re-entrant corner at $(x,y)=(0.5,0.5)$.
	Therefore, approximate solutions constructed with only a finite element basis may not lead to sufficiently small residuals.
	We can obtain smaller residuals by constructing approximates solutions with the sum of finite element basis functions and a singular function of the form $r^{\frac{2}{3}} \sin \left(\frac{2}{3}\theta\right)$, where $(r,\theta)$ is a polar coordinate system centered at the re-entrant corner.
	In {\rm \cite{kobayashi2009constructive}}, a priori error estimates for such a singular function are discussed.
	We also quote {\rm\cite[Example 7.7]{nakaoplumwatanabe2019numerical}}, an example of application to nonlinear elliptic problems.
	In this paper, we do not use the above singular function, but instead make the size of meshes around $(x,y)=(0.5,0.5)$ smaller than others to reduce residuals $($again, see, Fig.~{\rm\ref{fig:Lshape}}$)$.
	Once an $H^1_0$-error estimation of a solution is obtained, even if it is relatively rough, our method can be effective in proving the positivity of the solution.
	
%	Even with relatively rough $H^1_0$-error estimations, our method works well to prove the positivity of solutions.
%	For example, in \cite{kobayashi2009constructive}, 

%	using a singular function in the form $r^{\frac{2}{3}} \sin \left(\frac{2}{3} \theta\right)$ which expresses the singularity.
	% of the solution enables us to obtain smaller residuals.
	
%	Therefore, an attempt has been made to improve the accuracy of the solution by simultaneously using a function that expresses the singularity of the solution.

%	Instead, the mesh size around the non-convex corners is set to be smaller than the others.
%	Even with this relatively rough error evaluation, our method works well to prove the positivity of the solution.
%	$\gamma(r, \varphi):=r^{\frac{2}{3}} \sin \left(\frac{2}{3} \varphi\right)$
%	\cite{kobayashi2009constructive} and \cite[Example 7.7]{nakaoplumwatanabe2019numerical}
%	and
%	\cite{grisvard2011elliptic}		
	\end{rem}
	
	\section{Numerical Experiments}\label{sec:ex}
	In this section, we present numerical experiments in which the positivity of solutions of \eqref{main:fpro} satisfying \eqref{eq:h10error} are verified via the proposed method.
	All computations were implemented on a computer with 2.90 GHz Intel Xeon Platinum 8380H CPUs $\times$ 4, 3 TB RAM, and CentOS 8.2 using MATLAB 2019a with GCC Version 8.3.1.
	%Intel(R) Xeon(R) Platinum 8380H CPU @ 2.90GHz $\times$4	3T RAM
		%All computations were implemented on a computer with 
		%*****************************************
		%a 2.90 GHz Intel Core(TM) i9-7920X CPU, 128 GB RAM, and Ubuntu 18.04 using MATLAB 2019b with GCC version 6.3.0.
		All rounding errors were strictly estimated using the toolboxes INTLAB version 11 \cite{rump1999book} and kv library version 0.4.49 \cite{kashiwagikv}.
	In the tables in this section, we use the following notation:
	\begin{itemize}
		\item[$\cdot$] $M_u$: number of Legendre basis functions on $\Omega = (0,1)^2$ with respect to $x$ and $y$ for constructing approximate solution $\hat{u} \in V_{M_u}$ (see \eqref{eq:Legendrebase})
		\item[$\cdot$] $M$: number of Legendre basis functions on $\Omega = (0,1)^2$ with respect to $x$ and $y$ for calculating $\|F'^{-1}_{\hat{u}}\|_{\mathcal{L}(H^{-1},H^1_0)}$ (see \eqref{eq:Legendrebase})
		\item[$\cdot$] $\|F'^{-1}_{\hat{u}}\|$: operator norm  $\|F'^{-1}_{\hat{u}}\|_{\mathcal{L}(H^{-1},H^1_0)}$ required in \eqref{ineq:alphabeta}
		\item[$\cdot$] $\|F(\hat{u})\|$: residual norm $\|F(\hat{u})\|_{H^{-1}}$ required in \eqref{ineq:alphabeta}
		\item[$\cdot$] $L$: upper bound for Lipschitz constant satisfying \eqref{ineq:lip}
		\item[$\cdot$] $\alpha$ and $\beta$: constants required in Theorem \ref{theo:nk} 
		\item[$\cdot$] $\rho$: error bound $\| u- \hat{u} \|_{H^1_0}$
		\item[$\cdot$] $m$: constant that determines $D(m)$; see Lemma \ref{lem:dm}
		\item[$\cdot$] $|\supp u_{-}|$: volume of support of $u_{-}$
		\item[$\cdot$] $\lambda_{1}(\supp u_{-})$: first eigenvalue of $-\Delta$ on interior of $\supp u_{-}$ defined by \eqref{eq:defi-diri-eigenvalue}
		\item[$\cdot$] $\mathcal{C}_1$ and $\mathcal{C}_2$: constants required in Theorem \ref{theo:main}
	\end{itemize}

	\subsection{Lane–Emden equation}
	As mentioned in the Introduction,
	positive solutions of the Lane–Emden equation with subcritical $p>1$,
	\begin{align}
	\left\{\begin{array}{l l}
	-\Delta u=u\left|u\right|^{p-1} &\mathrm{in~} \Omega,\\
	u=0 &\mathrm{on~} \partial\Omega
	\end{array}\right.\label{emden}
	\end{align}
	 have been studied from various points of view (again, see \cite{gidas1979symmetry,lin1994uniqueness,damascelli1999qualitative,gladiali2011bifurcation,de2019morse}).
	 This equation is covered by Theorem \ref{theo:main} (see the first row in Table \ref{table:concrete_nonlinearities}).
      The Lipschitz constant $ L $ satisfying \eqref{ineq:lip} can be estimated as
		\begin{align*}
%		\label{lip:emden}
		L \leq p(p-1) C_{p+1}^{3}\left(\|\hat{u}\|_{L^{p+1}}+C_{p+1} r\right)^{p-2},~~r=2\alpha+\delta~~{\rm for~small}~~\delta>0
		\end{align*}
      via a simple calculation from the definition, where we set $ r $ to be the next floating-point number after $ 2\alpha $.	 
      We refer to \cite[Section 4]{tanaka2020numerical} as a numerical experiment for positive solutions of the Lane–Emden equation for $p=3, 5$ on $\Omega=(0,1)^2$,
      where the volume of $D(m)$ was roughly estimated as $|D(m)|\leq |\Omega|$.
      %The checked condition for the positivity verification
      %corresponds to the first row in Table \ref{table:concrete_nonlinearities} with $|D(m)|$ replaced with $|\Omega|$.
      
      For the L-shaped domain $\Omega=(0,1)^2\backslash([0,0.5]\times[0.5,1])$,
      we constructed an approximate solution $ \hat{u} \in V_M $ of \eqref{emden} with $p=3$ using a piecewise quadratic basis, obtaining
      Fig.~\ref{fig:emden_L}.
      Table \ref{tab:emden_L} shows the verification result and confirms the positivity of the enclosed solution because $\mathcal{C}_1 \leq \mathcal{C}_2$. 
      %We saw that $\mathcal{C}_1 \leq \mathcal{C}_2$, which therefore confirmed t
      	
	\begin{figure}[h]
			\begin{center}
				\includegraphics[height=60 mm]{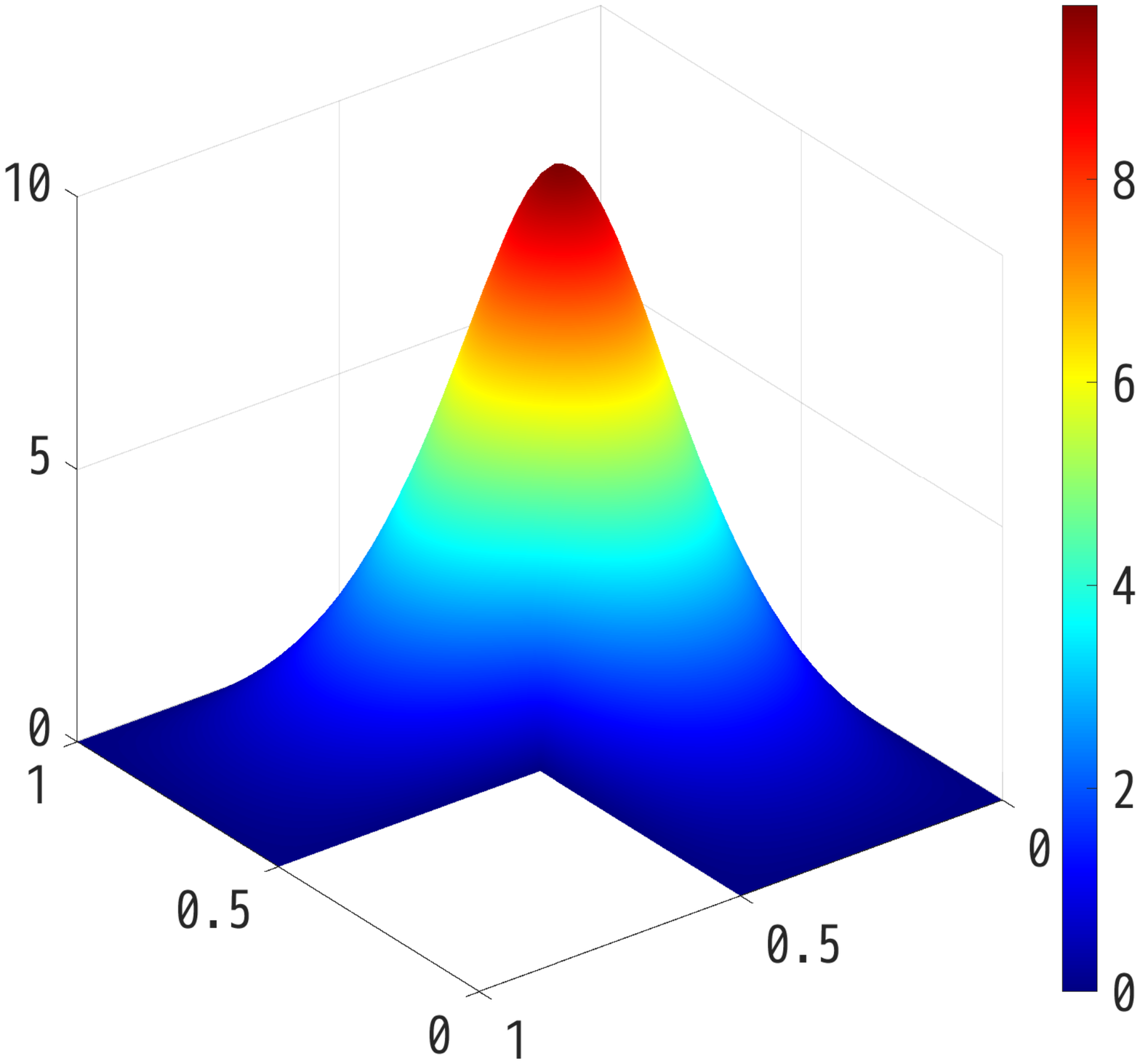}
			\end{center}
		\caption{An approximate solution of \eqref{emden} for $p=3$ on $\Omega=(0,1)^2\backslash([0,0.5]\times[0.5,1])$.}
		\label{fig:emden_L}
	\end{figure} 
	
	\begin{table}[h]
	\caption{Verification results for \eqref{emden} for $p=3$ on $\Omega=(0,1)^2\backslash([0,0.5]\times[0.5,1])$.
	The values in rows from $\|F'^{-1}_{\hat{u}}\|$ to $\rho$ in row $\mathcal{C}_1$ represent strict upper bounds in decimal form.
	The value in row $|\supp u_{-}|$ represents strict lower bounds in decimal form.}
	\label{tab:emden_L}
	\begin{center}
		\renewcommand\arraystretch{1.3}
		\footnotesize
		\begin{tabular}{ll}
			\hline
%			$p$& 3\\
%			\hline
%			\hline
			$\|F'^{-1}_{\hat{u}}\|$&$2.03035227$\\
			$ \|\mathcal{F}(\hat{u})\| $&$4.24332160\ten{-2}$\\
			$ L $&$7.11899016\ten{-1}$\\
			$\alpha$&$ 8.61543762\ten{-2}$\\
			$\beta$&$ 1.44540578$\\
			$\rho$&$ 1.03811119\ten{-1}$\\
			$m$&$ 2^{-4}$\\
			$|\supp u_{-}| \geq$ & $8.74678937\ten{-2}$\\
			$\mathcal{C}_1$ & $3.858891\ten{-3}$\\			
			$\mathcal{C}_2$ & $1$\\			
			\hline
		\end{tabular}
	\end{center}
	\end{table}
	
	\subsection{Allen–Cahn equation and Nagumo equation}
	In the next example, we consider the stationary problem of the Allen–Cahn equation
		\begin{align}
		\left\{\begin{array}{l l}
		-\Delta u=\lambda (u-u^3) &\mathrm{in~} \Omega,\\
		u=0 &\mathrm{on~} \partial\Omega
		\end{array}\right.\label{pro:allen}
		\end{align}
		with $\lambda>0$.
	This is regarded as a special case of the stationary problem of the Nagumo equation
	%we consider the stationary problem of the Nagumo equation
	\begin{align}
		\left\{\begin{array}{l l}
			-\Delta u=\lambda u (1-u) (u-a)= \lambda(-au+(1+a)u^2-u^3) &\mathrm{in~} \Omega,\\
			u=0 &\mathrm{on~} \partial\Omega
		\end{array}\right.\label{pro:nagumo}
	\end{align}
	with $\lambda>0$ and $0<a<1$.
	By applying $u=(v+1)/2$ to \eqref{pro:nagumo} with $a=0.5$ and adjusting the value of $\lambda$, these equations become identical.
%	By setting $a=0.5$ and applying $u=(v+1)/2$
%	$u=(v+1)/2, a=0.5$
	The Lipschitz constants $ L $ satisfying \eqref{ineq:lip} for \eqref{pro:allen} and \eqref{pro:nagumo} were respectively estimated as
	\begin{align*}
		&L \leq 6 \lambda C_{4}^{3}\left(\|\hat{u}\|_{L^{4}}+C_{4} r\right),&~~r=2\alpha+\delta~~{\rm for~small}~~\delta>0,\\
		&L \leq \lambda  \left( 2 (1+a) C_{3}^{3} + 6  	C_{4}^{3}\left(\|\hat{u}\|_{L^{4}}+C_{4} r\right) \right),&~~r=2\alpha+\delta~~{\rm for~small}~~\delta>0,
	\end{align*}
	where we set $ r $ to be the next floating-point number after $ 2\alpha $.
%No positive solution is admitted when $\varepsilon^{-2} < \lambda_1(\Omega)$.
%This can be confirmed by multiplying $-\Delta u = \varepsilon^{-2}(u-u^3)$ with the first eigenfunction of $-\Delta$ and integrating both sides.
%For a sign-changing solution $u$, let $\Omega'$ be a positive nodal domain of $u$.
%Note that $-u$ is also a solution of \eqref{main:fpro}, and therefore, considering only positive nodal domains is sufficient.
%The restricted function $u_{\Omega'}$ is a solution of a zero-Dirichlet problem restricted on $\Omega'$ and $ \lambda_1(\Omega) \leq \lambda_1(\Omega')$.
%Thus, if $\varepsilon^{-2} < \lambda_1(\Omega) (\leq \lambda_1(\Omega'))$, $u$ is the trivial solution.
	It should be again noted that, in the previous paper \cite{tanaka2020numerical}, another approach was used for \eqref{pro:allen} to confirm positivity,
	requiring the confirmation of the positivity of the minimal eigenvalue of a certain linearized operator around approximation $ \hat{u} $.
	%, requiring a nontrivial explicit evaluation
%	Besides, to prove the positivity of solutions of \eqref{pro:lions}
%	the previous method in \cite{tanaka2020numerical} requires an $ L^{\infty} $-error estimation \eqref{eq:linferror} and therefore is applicable to \eqref{pro:lions} only in the special cases where the solution has $H^2$-regularity and we can obtain an explicit bound for the embedding $ H^2(\Omega) \hookrightarrow L^{\infty}(\Omega) $ successfully.
	Theorem \ref{theo:main} can be uniformly applied to the nonlinearities of the form \eqref{eq:polyf} (again, see Table \ref{table:concrete_nonlinearities}).
	
%	 and \eqref{emden};
%	indeed, nonlinearity $f$ of \eqref{pro:allen} corresponds to the second row in Table \ref{table:concrete_nonlinearities}.
	
	For the square domain $\Omega=(0,1)^2$, we constructed approximate solutions $ \hat{u} $ using a Legendre polynomial basis.
	For \eqref{pro:allen} ($\lambda=100$, $400$, and $1600$), we obtained Fig.~\ref{fig3} and the verification results in Table \ref{tab3}.
	For \eqref{pro:nagumo} ($\lambda=400$, $a=1/4$), we found multiple solutions displayed in Fig.~\ref{fig:nagumo} and the verification results in Table \ref{tab:nagumo}.
	%	, and Fig.~\ref{fig:nagumo} and Table \ref{tab:nagumo} for \eqref{pro:nagumo} ($\lambda=400$, $a=1/4$).
	%, obtaining Fig.~\ref{fig3} and the verification results in Table \ref{tab3}.
	
	For the L-shaped domain $\Omega=(0,1)^2\backslash([0,0.5]\times[0.5,1])$,
	we constructed approximate solutions $ \hat{u} \in V_M $ using a piecewise quadratic basis, obtaining Fig.~\ref{fig:allen_nagumo_L} for \eqref{pro:allen} ($\lambda=100$) and \eqref{pro:nagumo} ($\lambda=180$, $a=1/64$), and the verification results in Table \ref{tab:allen_nagumo_L}.	
%	of \eqref{pro:allen} with $\lambda=100$ using a piecewise quadratic basis, obtaining Fig.~\ref{fig:emden_L} and the verification results in Table \ref{tab:allen_L}.
	In all cases, we confirmed $\mathcal{C}_1 \leq \mathcal{C}_2$ and thus the positivity of the enclosed solutions. 

	\newcommand{\sizee}{0.325\hsize}
	\begin{figure}[h]
		\begin{minipage}{\sizee}
			\begin{center}
				\includegraphics[height=33 mm]{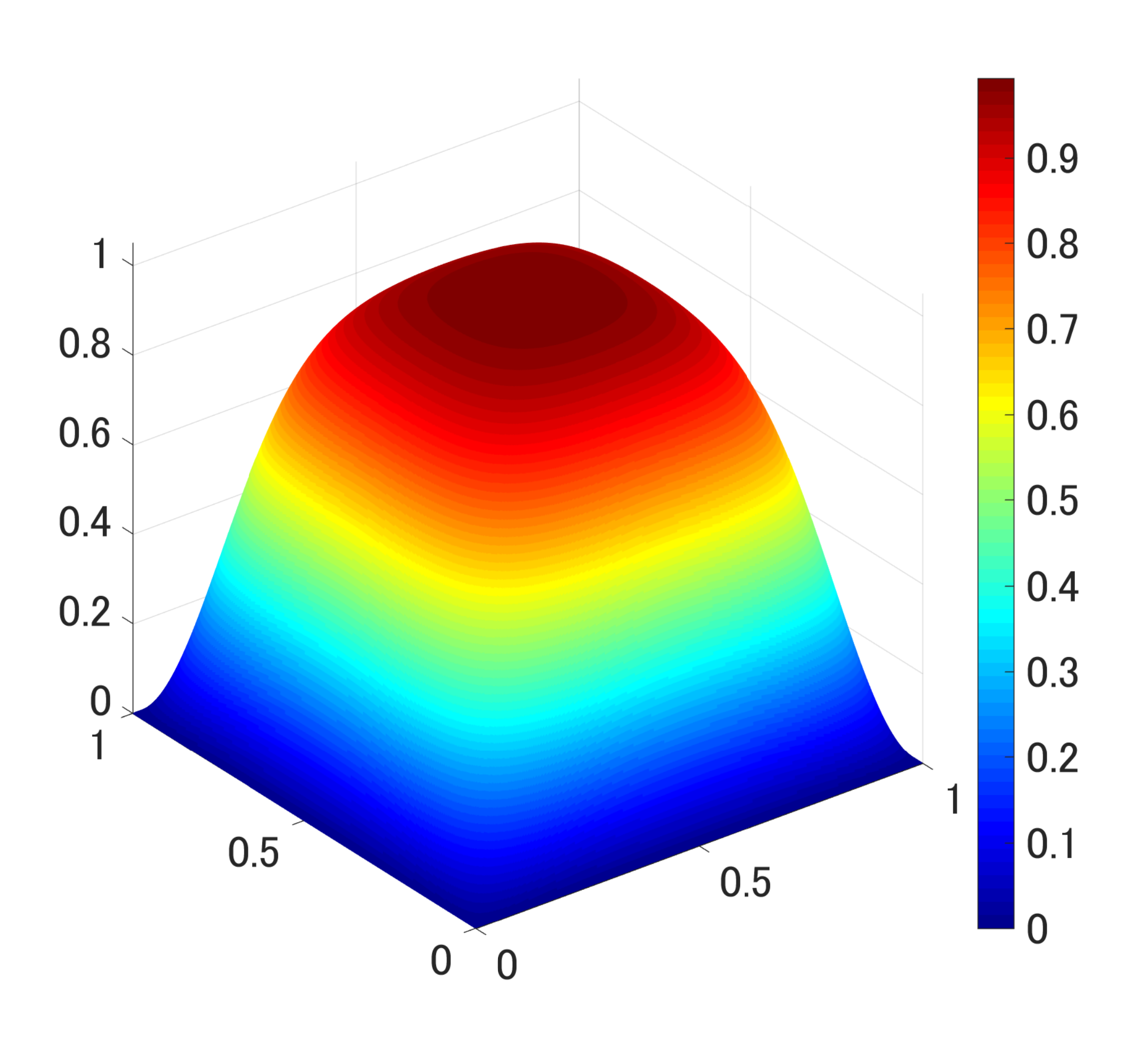}\\
				\footnotesize{$\lambda=100$}
			\end{center}
		\end{minipage}
		\begin{minipage}{\sizee}
			\begin{center}
				\includegraphics[height=33 mm]{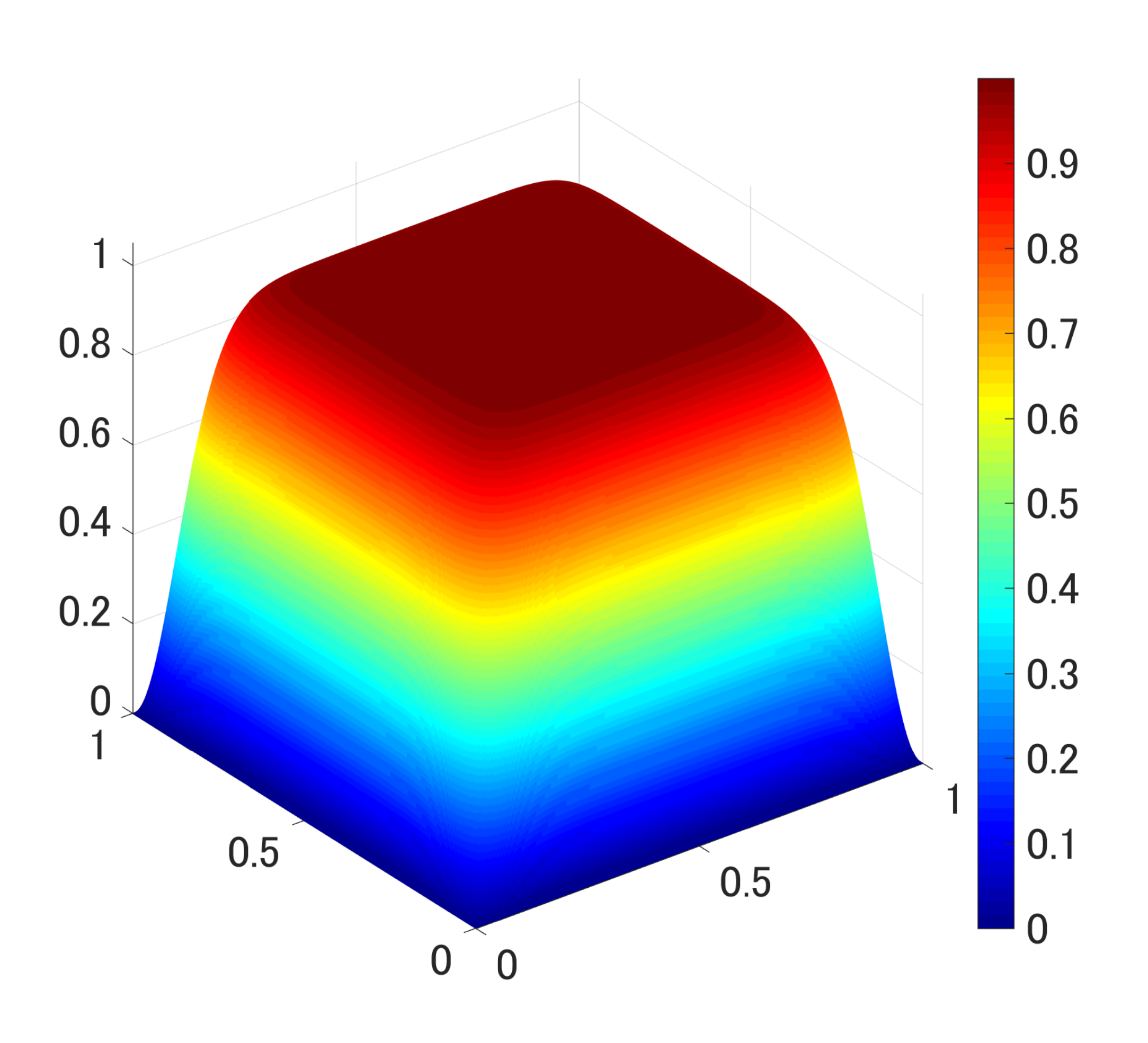}\\
				\footnotesize{$\lambda=400$}
			\end{center}
		\end{minipage}
		\begin{minipage}{\sizee}
			\begin{center}
				\includegraphics[height=33 mm]{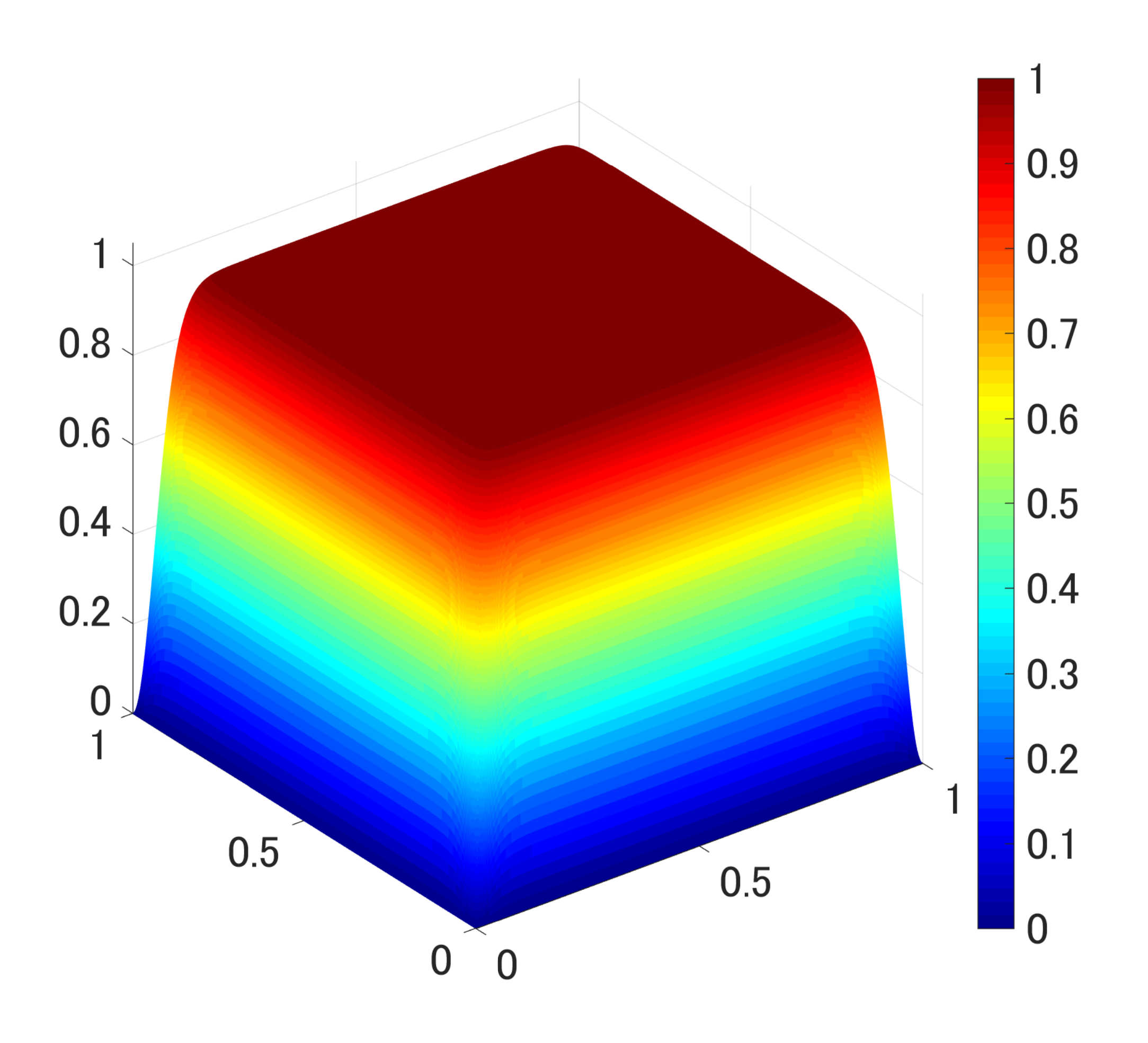}\\
				\footnotesize{$\lambda=1600$}
			\end{center}
		\end{minipage} 
		\caption{Approximate solutions of \eqref{pro:allen} on $\Omega=(0,1)^{2}$ for $\lambda=100$, $400$, and $1600$.}
		\label{fig3}
	\end{figure}

	\renewcommand{\sizee}{0.5\hsize}
	\begin{figure}[h]
		\begin{minipage}{\sizee}
			\begin{center}
				\includegraphics[height=40 mm]{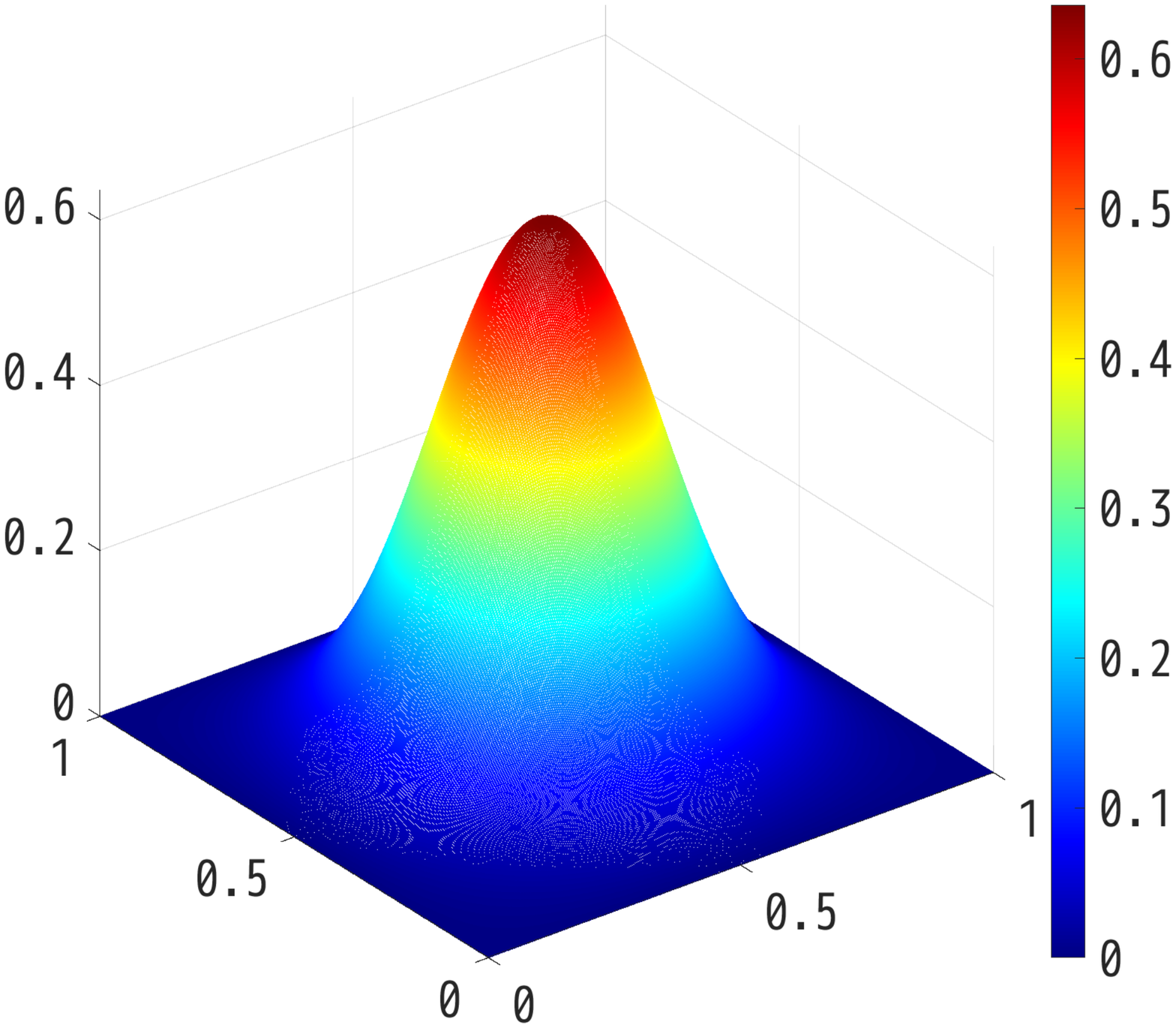}\\
				\footnotesize{Lower solution}
			\end{center}
		\end{minipage}
		\begin{minipage}{\sizee}
			\begin{center}
				\includegraphics[height=40 mm]{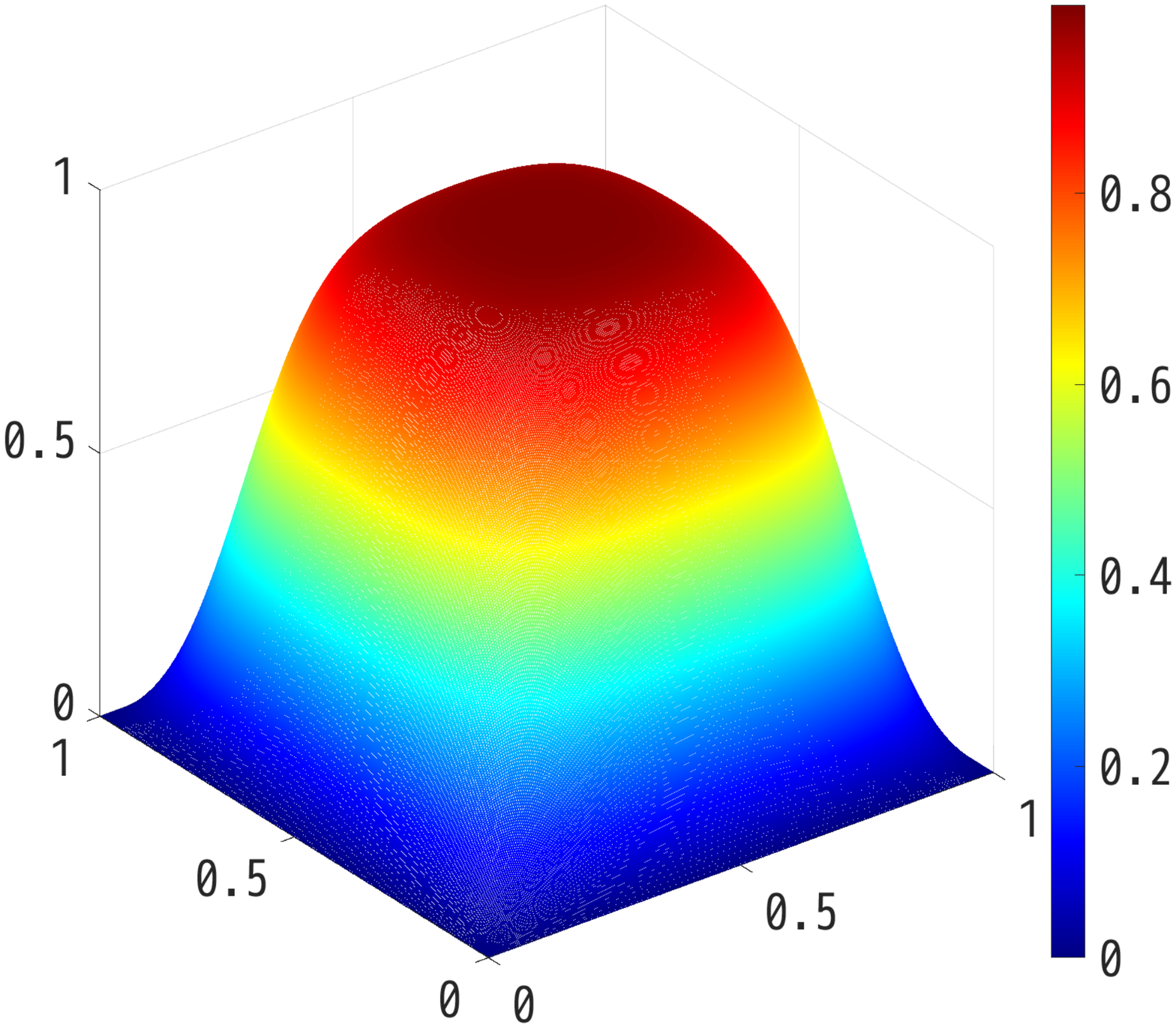}\\
				\footnotesize{Upper solution}
			\end{center}
		\end{minipage}
		\caption{Approximations of multiple solutions of \eqref{pro:nagumo} on $\Omega=(0,1)^{2}$ for $\lambda=400$ and $a=1/4$.}
		\label{fig:nagumo}
	\end{figure}
	
	\begin{table}[h]
	\caption{Verification results for \eqref{pro:allen} on $\Omega=(0,1)^{2}$ for $\lambda=100$, $400$, and $1600$.
	The values in rows from $ \|\mathcal{F}_{\hat{u}}^{\prime-1}\| $ to $\rho$ represent strict upper bounds in decimal form.
	The values in rows $\lambda_{1}(\supp u_{-})$ and $\mathcal{C}_2$ represent strict lower bounds in decimal form.}
	\label{tab3}
	\begin{center}
		\renewcommand\arraystretch{1.3}
		\footnotesize
%		\begin{tabular}{llll}
%			\hline
%			$\lambda$& 100& 400&1600\\
%			\hline
%			\hline
%			$M_u$& 40 & 40 & 60\\
%			$M$& 40 & 40 & 60\\
%			$ \|\mathcal{F}_{\hat{u}}^{\prime-1}\| $&$2.81346407$&$4.57367687$&$26.8136948$\\
%			$ \|\mathcal{F}(\hat{u})\| $&$5.62890577\ten{-10}$&$2.15869521\ten{-6}$&$1.99428443\ten{-6}$\\
%			$ L $&$3.05079436$&$5.02704780$&$7.57229901$\\
%			$\alpha$&$ 1.58367241\ten{-9}$&$ 9.87317430\ten{-6}$&$5.34741338\ten{-5}$\\
%			$\beta$&$ 8.58330029$&$ 22.9920922$&$2.03041315\ten{+2}$\\
%			$\rho$&$ 1.76247606\ten{-9}$&$	1.78014183\ten{-5}$&$5.68280000\ten{-4}$\\
%			$m$&$ 2^{-4}$&$	2^{-5}$&$2^{-5}$\\
%			$\lambda_{1}(\supp u_{-}) \geq$ & $1.23212099\ten{+2}$&$7.45293815\ten{+2}$&$6.12304584\ten{+3}$\\
%			$\mathcal{C}_1$ & $0$&$0$&$0$\\			
%			$\mathcal{C}_2\geq$ & $1.883914\ten{-1}$&$4.632989\ten{-1}$&$7.386921\ten{-1}$\\			
%			\hline
%		\end{tabular}
		\begin{tabular}{llll}
			\hline
			$\lambda$& 100& 400&1600\\
			\hline
			\hline
			$M_u$& 40 & 40 & 60\\
			$M$& 40 & 40 & 60\\
			$ \|\mathcal{F}_{\hat{u}}^{\prime-1}\| $&$2.81346407$&$4.57367687$&$26.8136948$\\
			$ \|\mathcal{F}(\hat{u})\| $&$5.62890577\ten{-10}$&$2.15869521\ten{-6}$&$1.99428443\ten{-6}$\\
			$ L $&$3.05079436$&$5.02704780$&$7.57229901$\\
			$\alpha$&$ 1.58367241\ten{-9}$&$ 9.87317430\ten{-6}$&$5.34741338\ten{-5}$\\
			$\beta$&$ 8.58330029$&$ 22.9920922$&$2.03041315\ten{+2}$\\
			$\rho$&$ 1.76247606\ten{-9}$&$	1.78014183\ten{-5}$&$5.68280000\ten{-4}$\\
			$m$&$ 2^{-4}$&$	2^{-5}$&$2^{-5}$\\
			$\lambda_{1}(\supp u_{-}) \geq$ & $3.24128275\ten{+2}$&$1.26466832\ten{+3}$&$2.07075690\ten{+3}$\\
			$\mathcal{C}_1$ & $0$&$0$&$0$\\			
			$\mathcal{C}_2\geq$ & $6.914802\ten{-1}$&$6.837115\ten{-1}$&$2.273357\ten{-1}$\\			
			\hline
		\end{tabular}
	\end{center}
	\end{table}

\begin{table}[h]
	\caption{Verification results for \eqref{pro:nagumo} on $\Omega=(0,1)^{2}$ for $\lambda=400$, $a=1/4$.
		The values in rows from $ \|\mathcal{F}_{\hat{u}}^{\prime-1}\|$ to $\rho$ and in row $\mathcal{C}_1$ represent strict upper bounds in decimal form.
		The values in row $\lambda_{1}(\supp u_{-})$ represent strict lower bounds in decimal form.
	}
	\label{tab:nagumo}
	\begin{center}
		\renewcommand\arraystretch{1.3}
		\footnotesize
		\begin{tabular}{lll}
			\hline
			Solution& Lower& Upper\\
			\hline
			\hline
			$M_u$& 40 & 40\\
			$M$& 40 & 80\\
			$ \|\mathcal{F}_{\hat{u}}^{\prime-1}\| $&$14.7881126$&$20.9590503$\\
			$ \|\mathcal{F}(\hat{u})\| $&$2.89599073\ten{-8}$&$7.71958696\ten{-8}$\\
			$ L $&$45.4341203$&$81.6094700$\\
			$\alpha$&$ 4.28262367\ten{-7}$&$ 1.61795212\ten{-6}$\\
			$\beta$&$ 6.71884883\ten{2}$&$ 1.71045699\ten{3}$\\
			$\rho$&$ 4.28323999\ten{-7}$&$	1.62019712\ten{-6}$\\
			$m$&$ 2^{-4}$&$	2^{-4}$\\
			$|\supp u_{-}| \geq$ & $4.99000550\ten{-2}$ & $5.59616089\ten{-3}$\\
			%$\lambda_{1}(\supp u_{-}) \geq$ & $47.9974691$&$1.0229255\ten{+2}$\\
			$\mathcal{C}_1$&$2.910457\ten{-2}$&$ 5.157194\ten{-3}$\\
			$\mathcal{C}_2 $&$ 1$&$1$\\
			\hline
		\end{tabular}
	\end{center}
	
\end{table} 
	
	\renewcommand{\sizee}{0.5\hsize}
	\begin{figure}[h]
		\begin{minipage}{\sizee}
			\begin{center}
				\includegraphics[height=40 mm]{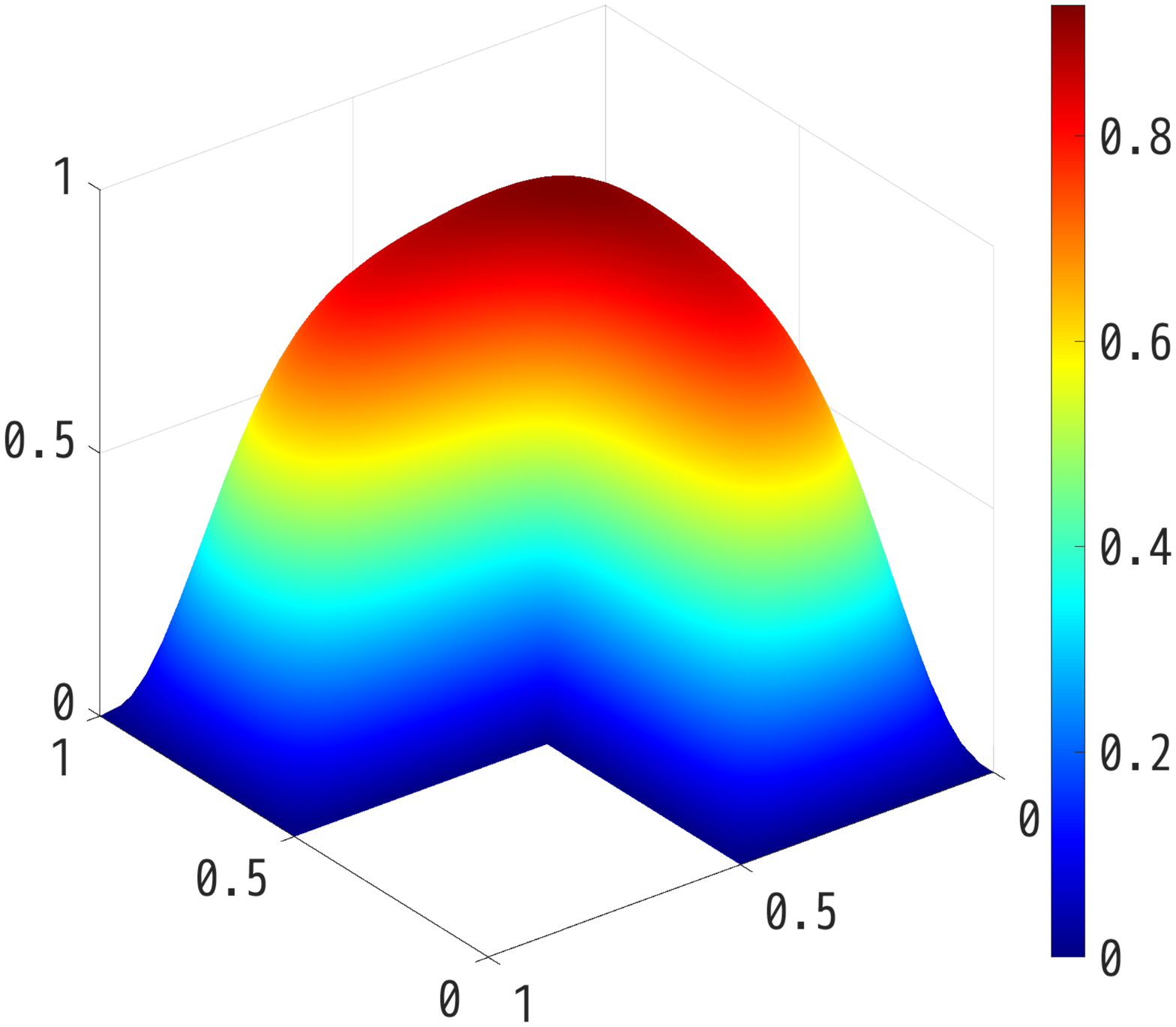}\\
				\footnotesize{\eqref{pro:allen} with $\lambda=100$}
			\end{center}
		\end{minipage}
		\begin{minipage}{\sizee}
			\begin{center}
				\includegraphics[height=40 mm]{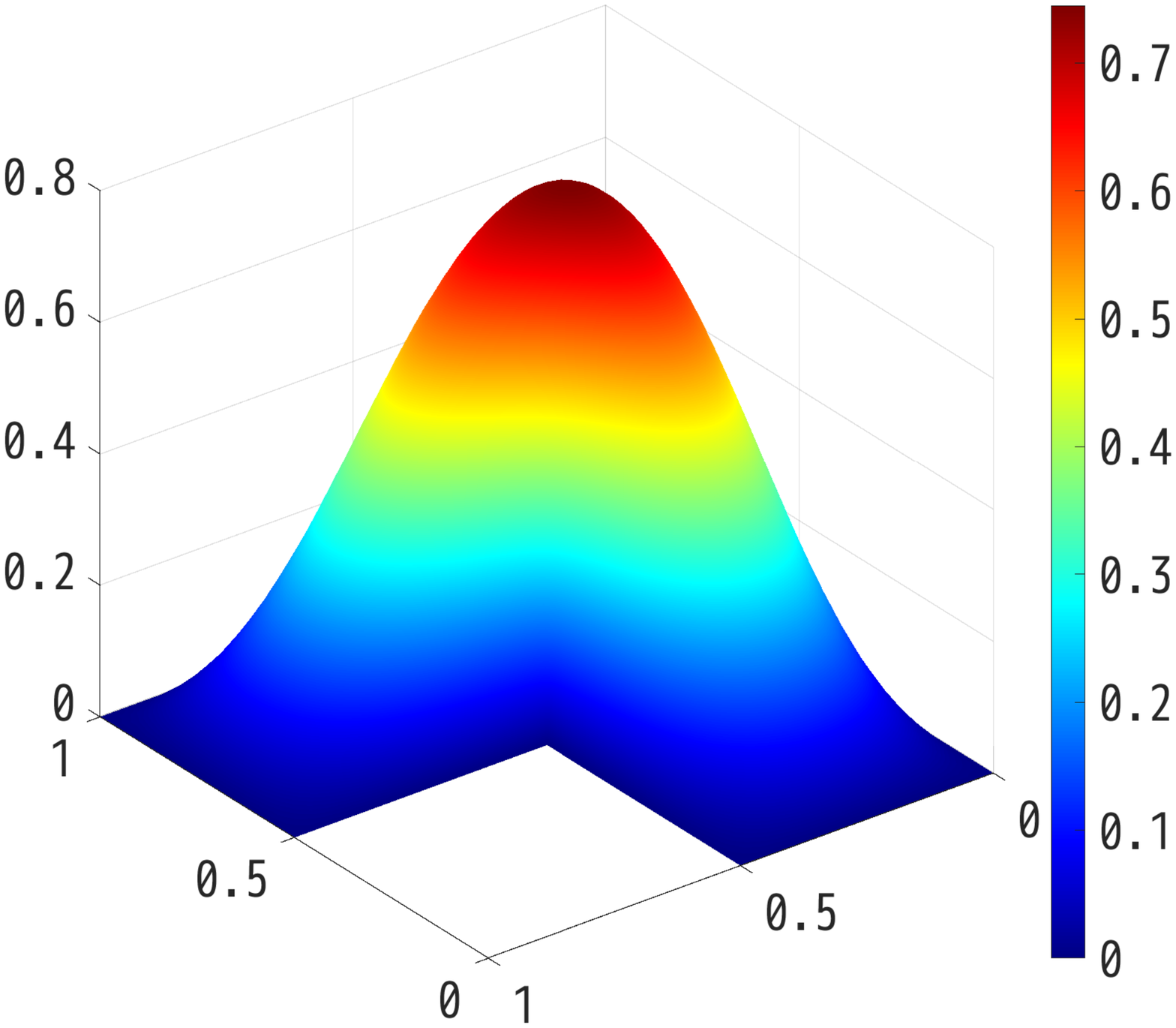}\\
				\footnotesize{\eqref{pro:nagumo} with $\lambda=180$, $a=1/64$}
			\end{center}
		\end{minipage}
		\caption{Approximate solutions of \eqref{pro:allen} and \eqref{pro:nagumo} on $\Omega=(0,1)^2\backslash([0,0.5]\times[0.5,1])$.}
		\label{fig:allen_nagumo_L}
	\end{figure}

%		\begin{figure}[h]
%				\begin{center}
%					\includegraphics[height=50 mm]{L_allen_lam100.pdf}\\
%				\end{center}
%				~
%			\caption{An approximate solution of \eqref{pro:allen} on $\Omega=(0,1)^2\backslash([0,0.5]\times[0.5,1])$ for $\lambda=100$.}
%			\label{fig:allen_L}
%		\end{figure} 
%	
	\begin{table}[h]
	\caption{Verification results for \eqref{pro:allen} and \eqref{pro:nagumo} on $\Omega=(0,1)^2\backslash([0,0.5]\times[0.5,1])$.
	The values in rows from $ \|\mathcal{F}_{\hat{u}}^{\prime-1}\| $ to $\rho$ represent strict upper bounds in decimal form.
	The values in row $\lambda_{1}(\supp u_{-})$ and $\mathcal{C}_2$ represent strict lower bounds in decimal form. The value of $\mathcal{C}_2$ for \eqref{pro:nagumo} is exactly 1.
}
	\label{tab:allen_nagumo_L}
	\begin{center}
		\renewcommand\arraystretch{1.3}
		\footnotesize
		\begin{tabular}{lll}
			\hline
			Problem& \eqref{pro:allen} ($\lambda=100$)& \eqref{pro:nagumo} ($\lambda=180$, $a=1/64$)\\
			\hline
			\hline
			$ \|\mathcal{F}_{\hat{u}}^{\prime-1}\| $&$1.45330579$ &$2.11075231$\\
			$ \|\mathcal{F}(\hat{u})\| $&$8.04913524\ten{-3}$ &$4.93049851\ten{-3}$\\
			$ L $&$10.1025501$ &$18.0081337$\\
			$\alpha$&$ 1.1697855\ten{-2}$ &$ 1.04070611\ten{-2}$\\
			$\beta$&$ 14.6820944$ &$ 38.0107096$\\
			$\rho$&$ 1.39688500\ten{-2}$ &$ 1.51162215\ten{-2}$\\
			$m$&$ 2^{-4}$ &$ 2^{-4}$\\
			$\lambda_{1}(\supp u_{-}) \geq$ & $2.12873779\ten{+3}$ & $8.70246701\ten{-2}$\\
			$\mathcal{C}_1$ & $0$ & $1.920365\ten{-1}$\\			
			$\mathcal{C}_2$ & $\geq 9.530238\ten{-1}$ & $=1$\\			
			\hline
		\end{tabular}
	\end{center}

	\end{table}

	\subsection{Elliptic equation with multiple positive solutions}
	For the last example, we consider the elliptic boundary value problem
		\begin{align}
		\left\{\begin{array}{l l}
		-\Delta u=\lambda (u+\lionsA u^2 - \lionsB u^3) &\mathrm{in~} \Omega,\\
		u=0 &\mathrm{on~} \partial\Omega,
		\end{array}\right.\label{pro:lions}
		\end{align}
	given $\lambda, \lionsA,\lionsB>0$.
%	The bifurcation of problem \eqref{eq:mainpro} with this nonlinearity was analyzed in \cite{lions1982existence}.
%	This problem has two positive solutions when $ \lambda^*< \lambda < \lambda_1(\Omega)$, where $\lambda^*$ is some positive constant.
	This problem has two positive solutions when $ \lambda^*< \lambda < \lambda_1(\Omega)$ for a certain $\lambda^*>0$ (see \cite{lions1982existence}).
	The Lipschitz constant $ L $ for this problem can be estimated as
		\begin{align*}
		L \leq \lambda  \left( 2 \lionsA C_{3}^{3} + 6 \lionsB C_{4}^{3}\left(\|\hat{u}\|_{L^{4}}+C_{4} r\right) \right),~~r=2\alpha+\delta~~{\rm for~small}~~\delta>0,
		\end{align*}
	where we again set $ r $ to be the next floating-point number of $ 2\alpha $.
	
	To prove the positivity of solutions of \eqref{pro:lions}, the previous method \cite{tanaka2020numerical} requires an $ L^{\infty} $-error estimation \eqref{eq:linferror}
	% and further numerical computations in addition to calculating $ H^1_0 $-error $\rho$.
	 and therefore is applicable to \eqref{pro:lions} only in the special cases where the solution has $H^2$-regularity and we can obtain an explicit bound for the embedding $ H^2(\Omega) \hookrightarrow L^{\infty}(\Omega) $ successfully.
	However, the proposed method is well applicable even to \eqref{pro:lions} without assuming $H^2$-regularity;
	%\eqref{pro:lions} when assuming an $ H^1_0 $-error estimation only;
	see again the last case in Table \ref{table:concrete_nonlinearities}.

	For the square domain $\Omega=(0,1)^2$, we constructed approximations $\hat{u}$ of multiple positive solutions of \eqref{pro:lions} ($\lambda=10$, $\lionsA=5$, $\lionsB=1$) using a Legendre polynomial basis, obtaining the results in Fig.~\ref{fig:lions} and Table \ref{tab:lions}.
	%Table \ref{tab:lions} shows the verification results for the solutions.
	For the L-shaped domain $\Omega=(0,1)^2\backslash([0,0.5]\times[0.5,1])$,
	we constructed multiple approximate solutions $ \hat{u} \in V_M $ of \eqref{pro:lions} ($\lambda=20$, $\lionsA=5$, $\lionsB=1$) using a piecewise quadratic basis, obtaining Fig.~\ref{fig:lions_L} and Table \ref{tab:lions_L}.
	Since $\mathcal{C}_1 \leq \mathcal{C}_2$ holds for all cases, the positivity of the solutions are confirmed.
	Because solutions on the L-shaped domain have not $H^2$-regularity and problem \eqref{pro:lions} corresponds to the lower left case in Table \ref{table:applicable},
	the previous method \cite{tanaka2020numerical} cannot be applied to solutions in Fig.~\ref{fig:lions_L}.
	Nevertheless, the proposed method succeeded in proving the positivity of both lower and upper solutions.
	% according to $\mathcal{C}_1 \leq \mathcal{C}_2$.
%	Inequality $\mathcal{C}_1 \leq \mathcal{C}_2$ confirmed the positivity of both lower and upper solutions. 
		
	\renewcommand{\sizee}{0.5\hsize}
	\begin{figure}[h]
		\begin{minipage}{\sizee}
			\begin{center}
				\includegraphics[height=50 mm]{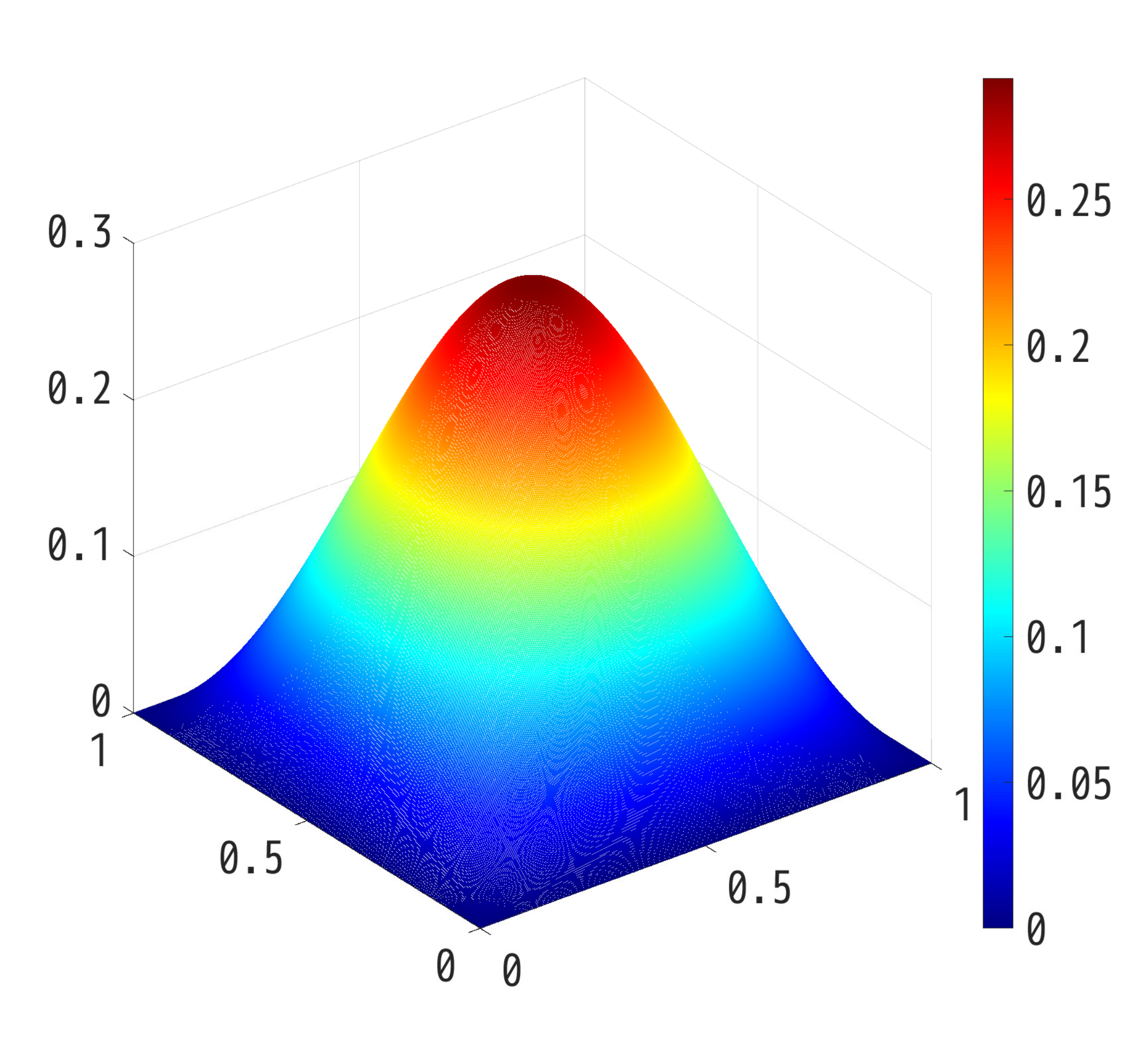}\\
				\footnotesize{Lower solution}
			\end{center}
			~
		\end{minipage}
		\begin{minipage}{\sizee}
			\begin{center}
				\includegraphics[height=50 mm]{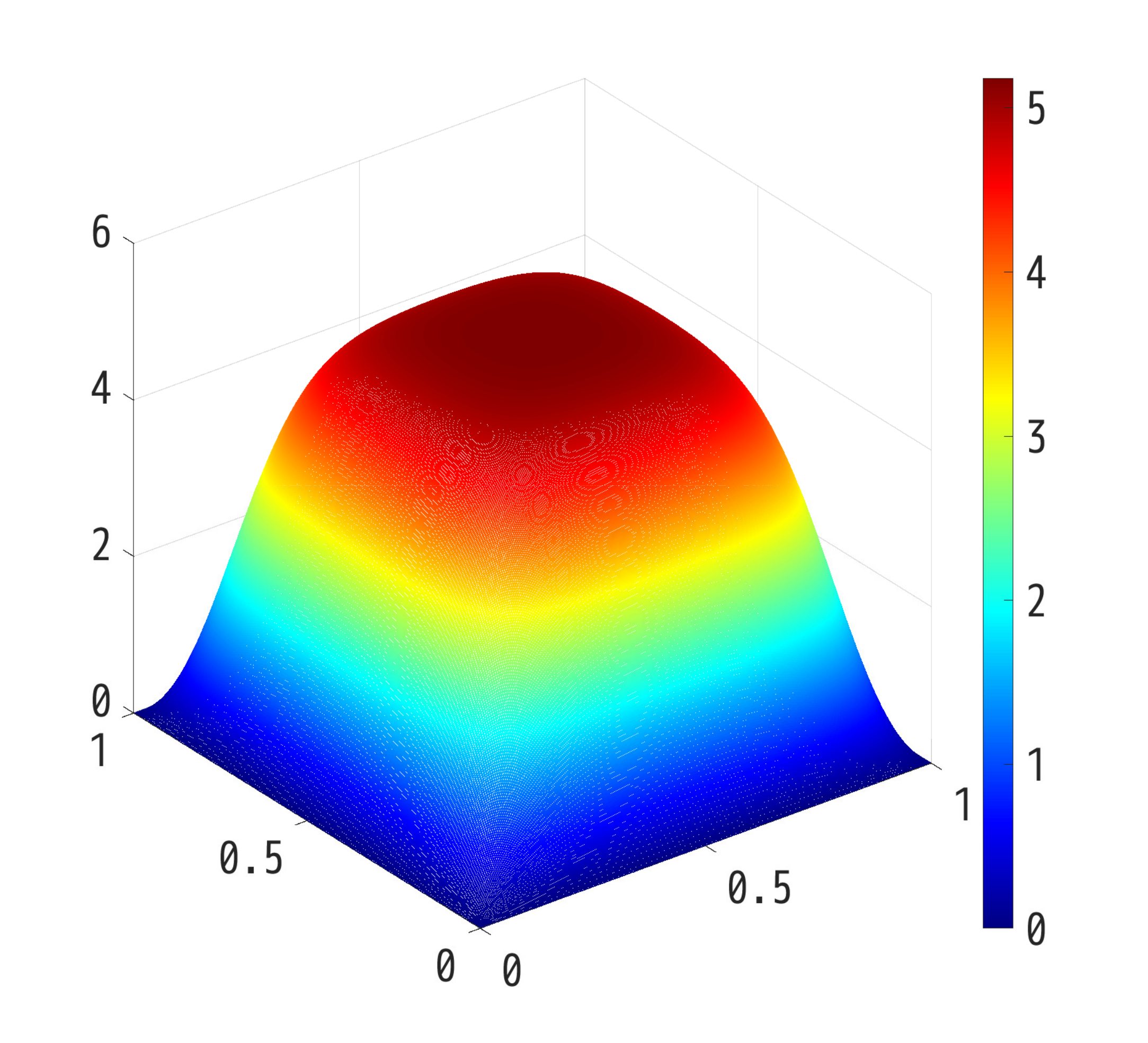}\\
				\footnotesize{Upper solution}
			\end{center}
			~
		\end{minipage}
		\caption{Approximations of multiple solutions of \eqref{pro:lions} on $\Omega=(0,1)^{2}$ when $\lambda=10$, $\lionsA=5$, $\lionsB=1$.}
		\label{fig:lions}
	\end{figure} 
	
	\begin{table}[h]
	\caption{Verification results for \eqref{pro:lions} on $\Omega=(0,1)^{2}$ when $\lambda=10$, $\lionsA=5$, $\lionsB=1$.
		The values in rows from $ \|\mathcal{F}_{\hat{u}}^{\prime-1}\| $ to $\rho$ and in row $\mathcal{C}_1$ represent strict upper bounds in decimal form.
		The values in rows $\lambda_{1}(\supp u_{-})$ and $\mathcal{C}_2$ represent strict lower bounds in decimal form.
		}
	\label{tab:lions}
	\begin{center}
		\renewcommand\arraystretch{1.3}
		\footnotesize
%		\begin{tabular}{lll}
%			\hline
%			Solution& Lower& Upper\\
%			\hline
%			\hline
%			$M_u$& 20 & 40\\
%			$M$& 20 & 80\\
%			$ \|\mathcal{F}_{\hat{u}}^{\prime-1}\| $&$2.11265000$&$13.7356291$\\
%			$ \|\mathcal{F}(\hat{u})\| $&$2.87319486\ten{-9}$&$5.76722406\ten{-8}$\\
%			$ L $&$2.52733502$&$2.77222751$\\
%			$\alpha$&$ 6.07005511\ten{-9}$&$ 7.92164505\ten{-7}$\\
%			$\beta$&$ 5.33937432$&$ 38.0782888$\\
%			$\rho$&$ 6.29824611\ten{-9}$&$	1.32442923\ten{-6}$\\
%			$m$&$ 2^{-4}$&$	2^{-4}$\\
%			$|\supp u_{-}| \geq$ & $4.18472291\ten{-3}$ & $4.12368775\ten{-3}$\\
%			$\lambda_{1}(\supp u_{-}) \geq$ & $16.7853836$&$6.13673372\ten{+2}$\\
%			$\mathcal{C}_1$&$1.073194\ten{-4}$&$5.516775\ten{-4}$\\
%			$\mathcal{C}_2 \geq$&$ 4.042436\ten{-1}$&$9.837047\ten{-1}$\\
%			\hline
%		\end{tabular}
		\begin{tabular}{lll}
			\hline
			Solution& Lower& Upper\\
			\hline
			\hline
			$M_u$& 20 & 40\\
			$M$& 20 & 80\\
			$ \|\mathcal{F}_{\hat{u}}^{\prime-1}\| $&$2.11265000$&$13.7356291$\\
			$ \|\mathcal{F}(\hat{u})\| $&$2.87319486\ten{-9}$&$5.76722406\ten{-8}$\\
			$ L $&$2.52733502$&$2.77222751$\\
			$\alpha$&$ 6.07005511\ten{-9}$&$ 7.92164505\ten{-7}$\\
			$\beta$&$ 5.33937432$&$ 38.0782888$\\
			$\rho$&$ 6.29824611\ten{-9}$&$	1.32442923\ten{-6}$\\
			$m$&$ 2^{-4}$&$	2^{-4}$\\
			$|\supp u_{-}| \geq$ & $4.18472291\ten{-3}$ & $4.12368775\ten{-3}$\\
			$\lambda_{1}(\supp u_{-}) \geq$ & $47.9974691$&$1.02292545\ten{+2}$\\
			$\mathcal{C}_1$&$1.081214\ten{-4}$&$7.964752\ten{-4}$\\
			$\mathcal{C}_2 \geq$&$ 7.916557\ten{-1}$&$9.902241\ten{-1}$\\
			\hline
		\end{tabular}
	\end{center}
	\end{table}

	\begin{figure}[h]
		\begin{minipage}{\sizee}
			\begin{center}
				\includegraphics[height=50 mm]{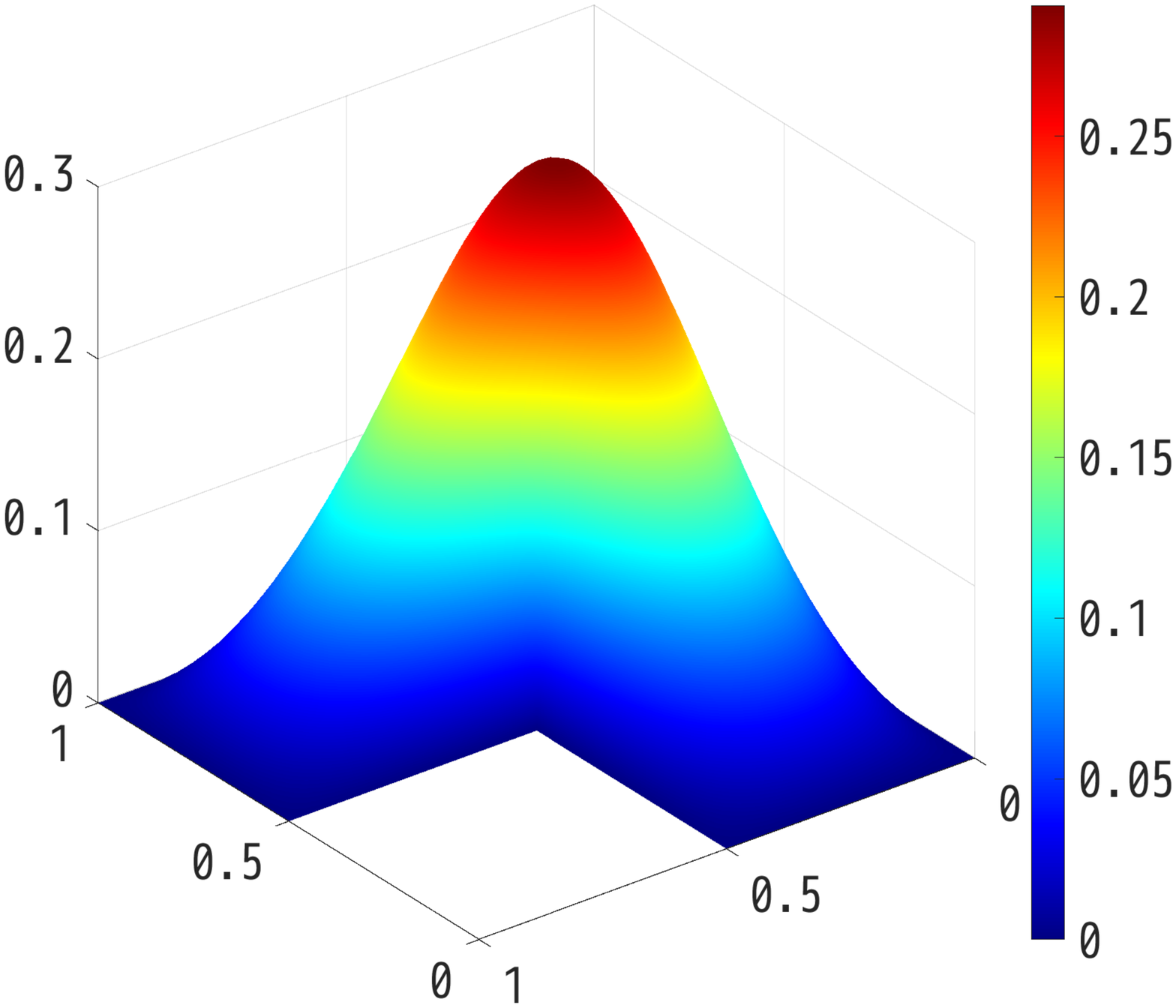}\\
				\footnotesize{Lower solution}
			\end{center}
			~
		\end{minipage}
		\begin{minipage}{\sizee}
			\begin{center}
				\includegraphics[height=50 mm]{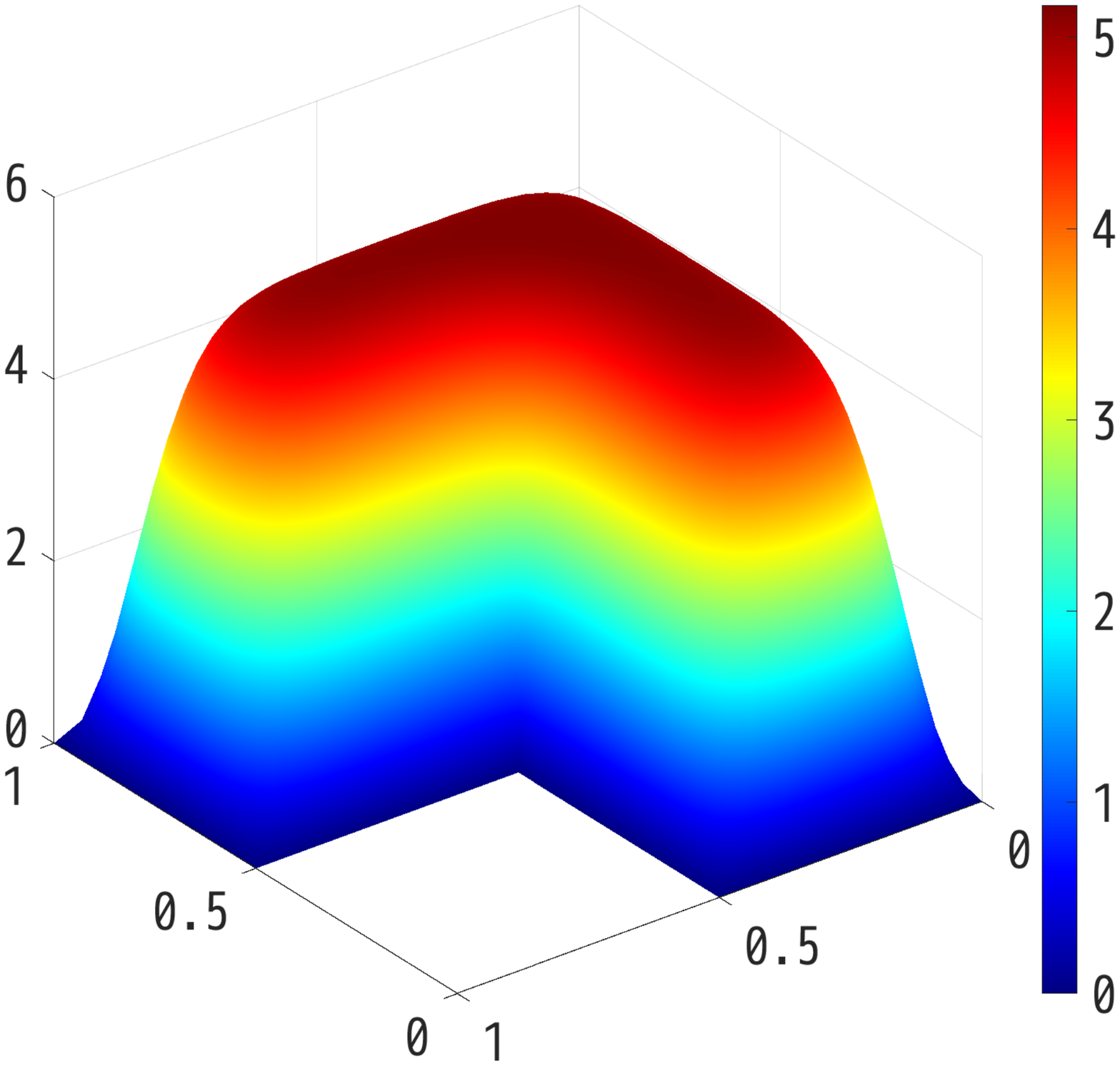}\\
				\footnotesize{Upper solution}
			\end{center}
			~
		\end{minipage}
		\caption{Approximations of multiple solutions of \eqref{pro:lions} on $\Omega=(0,1)^2\backslash([0,0.5]\times[0.5,1])$ when $\lambda=20$, $\lionsA=5$, $\lionsB=1$.}
		\label{fig:lions_L}
	\end{figure} 
	
		\begin{table}[h]
		\caption{Verification results for \eqref{pro:lions} on $\Omega=(0,1)^2\backslash([0,0.5]\times[0.5,1])$ when $\lambda=20$, $\lionsA=5$, $\lionsB=1$.
			The values in rows from $ \|\mathcal{F}_{\hat{u}}^{\prime-1}\| $ to $\rho$ and in row $\mathcal{C}_1$ represent strict upper bounds in decimal form.
			The values in rows $\lambda_{1}(\supp u_{-})$ and $\mathcal{C}_2$ represent strict lower bounds in decimal form.
			}
		\label{tab:lions_L}
		\begin{center}
			\renewcommand\arraystretch{1.3}
			\footnotesize
			\begin{tabular}{lll}
				\hline
				Solution& Lower& Upper\\
				\hline
				\hline
				$ \|\mathcal{F}_{\hat{u}}^{\prime-1}\| $&$3.93924470$&$4.17817762$\\
				$ \|\mathcal{F}(\hat{u})\| $&$1.75306001\ten{-3}$&$7.46385641\ten{-3}$\\
				$ L $&$3.78667858$&$3.83154427$\\
				$\alpha$&$ 6.90573233\ten{-3}$&$ 3.11853178\ten{-2}$\\
				$\beta$&$ 14.9166536$&$ 16.0088725$\\
				$\rho$&$ 7.38736922\ten{-3}$&$	6.0713451\ten{-2}$\\
				$m$&$ 2^{-4}$&$	2^{-2}$\\
				$|\supp u_{-}| \geq$ & $8.70246701\ten{-2}$ & $8.54343721\ten{-2}$\\
				$\lambda_{1}(\supp u_{-}) \geq$ & $56.5433134$&$6.85289494\ten{+4}$\\
				$\mathcal{C}_1$&$7.053746\ten{-2}$&$1.777890\ten{-2}$\\
				$\mathcal{C}_2 \geq$&$ 6.462889\ten{-1}$&$9.997082\ten{-1}$\\
				\hline
			\end{tabular}
		\end{center}
		\end{table} 

	\section{Conclusion}
		We proposed a unified a posteriori method for verifying the positivity of solutions $ u $ of elliptic boundary value problem \eqref{eq:mainpro} while assuming $ H^1_0 $-error estimation \eqref{eq:h10error} given some numerical approximation $ \hat{u} $ and an explicit error bound $ \rho $.
		By extending one of the approaches developed in \cite{tanaka2020numerical}, we designed a unified method with wide applicability.
		We described the way to obtain explicit values of several constants that the proposed method requires.
		We also presented numerical experiments to show the effectiveness of our method for three types of nonlinearities, including those to which the previous approach is not applicable.% without $L^{\infty}$-error estimations.
\appendix
\renewcommand{\thetheo}{\Alph{section}.\arabic{theo}}
	\section{Estimates of the first positive zeros of the Bessel functions}
	\label{appendix:rfkineq}
    This section discusses rigorous estimates of the Bessel functions $j_{n,1}$ required to obtain the values displayed in Table \ref{table:A1n}.
	When $n$ is written in the form $n=k+0.5$ with an integer $k$,
	explicit formulas for $j_{n,1}$ can be obtained.
	Particularly, $j_{0.5,1}=\sqrt{2/(\pi x)} \sin x$; therefore, the first zero of this is $\pi$ (see, for example, \cite[Remark 1.2]{baricz2010generalized} and \cite[Section 10.16]{olver2010nist}).
	Moreover, we have $j_{1.5,1}=\sqrt{2/(\pi x)} \left( x^{-1} \sin x - \cos x \right)$, the zeros of which satisfy $x=\tan x$.
	The first zero of this function is enclosed by the function ``allsol'' packaged in the kv library \cite{kashiwagikv}, which enables us to obtain all zeros of the function in a given compact interval.
	We set the initial interval as $[\pi,1.5\pi]$, in which the equation has the first zero.
	
	When $n=0,1$, we calculated rigorous values of $j_{n,1}$ using the bisection method with computer assistance.
	For preparation, we first prove that there is no positive zero of $j_{0,1}$ and $j_{1,1}$ in the compact interval $[0,1]$ as follows:
	When $n$ is an integer, $j_{n,1}$ is given by
	\begin{align*}
	    j_{n,1}(x) = \frac{1}{\pi} \int_{0}^{\pi} \cos(x\sin t-nt) dt;
	\end{align*}
	see \cite[Section 10.9]{olver2010nist}.
	Let $a\geq 0$ and suppose $0 \leq x \leq a$ so that $0 \leq x \sin t \leq a $ ($0\leq t \leq \pi$).
	Let us write $A:=[0,a]$.
	Then, we have 
	\begin{align*}
	    j_{0,1}(x) &= \frac{1}{\pi} \int_{0}^{\pi} \cos(x\sin t) dt \in \cos A = [\cos a, 1].
	\end{align*}
	When $a \in [0,1] \subset [0,\pi/2)$, $\cos a$ is positive, and therefore, so is $j_{0,1}$.
	Next, we consider $j_{1,1}$, which satisfies $j_{1,1}(0)=0$.
	The first derivative of $j_{1,1}$ is given by
	\begin{align*}
	    \frac{d}{dx}j_{1,1}(x) = \frac{1}{\pi} \int_{0}^{\pi} \sin{t}\sin(t-x\sin t) dt.
	\end{align*}
    Hence, we have, for $a\in [0,\pi]$,
	\begin{align*}
	    \frac{d}{dx}j_{1,1}(x) \in \frac{1}{\pi} \int_{0}^{\pi} \sin{t}\sin(t-A) dt
	    = \frac{1}{4 \pi} \left ( 2 \pi \cos A + \sin A - \sin A \right)
	    =:B
	    %\geq \frac{1}{4 \pi} \left ( 2 \pi \cos a - \sin a \right).
	\end{align*}
	For all $t \in B$, we confirm $t \geq \frac{1}{4 \pi} \left ( 2 \pi \cos a - \sin a \right)$.
	This value is positive when $a \in [0,1] \subset [0,\tan^{-1}(2 \pi))$.
	Therefore, $j_{n,1}$ monotonically increases for $x \in [0,1] $.
	
	Using bisection steps, we rigorously computed the first positive zeros of $j_{0,1}$ and $j_{1,1}$.
	We first found a compact interval that includes the first positive zero of $j_{n,1}$ by searching the first interval $[\underline{x},\overline{x}]:=[1+(i-1)\varepsilon/2,1+(i-1)\varepsilon/2+\varepsilon]$ ($i=1,2,\cdots$) that satisfies $j_{i,1}(\underline{x})j_{i,1}(\overline{x})<0$.
	Then, starting from the center of $[\underline{x},\overline{x}]$, we repeated the bisection method until the desired precision was achieved.
	All rounding errors were strictly estimated using kv library version 0.4.49 \cite{kashiwagikv}.
% \begin{acknowledgements}
% 	We express our profound gratitude to two anonymous referees for highly insightful comments and suggestions.
% \end{acknowledgements}
% BibTeX users please use one of
%\bibliographystyle{spbasic}      % basic style, author-year citations
\bibliographystyle{spmpsci}      % mathematics and physical sciences
\bibliography{ref}   % name your BibTeX data base
%% Non-BibTeX users please use
%\begin{thebibliography}{}
%%
%% and use \bibitem to create references. Consult the Instructions
%% for authors for reference list style.
%%
%\bibitem{RefJ}
%% Format for Journal Reference
%Author, Article title, Journal, Volume, page numbers (year)
%% Format for books
%\bibitem{RefB}
%Author, Book title, page numbers. Publisher, place (year)
%% etc
%\end{thebibliography}
\end{document}